\documentclass[12pt]{amsart}
\usepackage[margin=1in]{geometry}
\usepackage[english]{babel}
\usepackage{graphicx}
\usepackage{amsmath,amssymb,amsthm}
\usepackage{mathrsfs}
\usepackage[utf8]{inputenc}
\usepackage{cite}     
\usepackage{color}
\usepackage[unicode]{hyperref}
\hypersetup{colorlinks=true, citecolor=blue, linkcolor=blue, urlcolor=blue, pdfstartview=FitH, pdfauthor=Ali Ebadi, pdftitle=Generalized Shifts of Zeta Zeros and Partial Progress on Banks’ Conjecture}

\usepackage{hyperref} 

\usepackage{amssymb,amsfonts,amsmath}
\usepackage{color}
\usepackage{enumerate}
\usepackage{mathrsfs}
\usepackage[capitalise]{cleveref}
\usepackage{constants}

\newtheorem{theorem}{Theorem}[section]  
\newtheorem{lemma}[theorem]{Lemma}  

\begin{document}

\title{On the Zeros of the Riemann Zeta Function with Two Ordinate Shifts}

\author{Ali Ebadi}

\begin{abstract}

We prove that for any fixed real numbers $y_1, y_2 \neq 0$ and constant $C > 0$, there exists a threshold $T_{*} = T_{*}(y_1, y_2, C) > 0$ such that for all $T \geq T_{*}$, the interval \(
[T, T(1 + \epsilon)]
\), with $\epsilon = \exp\left(-C \sqrt{\log T}\right)$, contains at least one $\gamma$ satisfying 
\[
\zeta\left(\frac{1}{2} + i\gamma\right) = 0,\quad \zeta\left(\frac{1}{2} + i(\gamma + y_1)\right) \neq 0,\quad \text{and} \quad \zeta\left(\frac{1}{2} + i(\gamma + y_2)\right) \neq 0.
\]

This extends earlier work by Banks (for a single shift $y$) to two distinct shifts $y_1, y_2$. Our argument is based on the behavior of $\zeta$ and $L$ functions in zero-free regions via Perron’s formula.

\end{abstract}

\maketitle

\tableofcontents  

\newpage

\section{Introduction and statement of main results}
\label{sec:intro}

\subsection{Background} 

The Riemann zeta function $\zeta(s)$ plays a fundamental role in number theory, with its nontrivial zeros $\rho = \sigma + i\gamma$ conjectured to lie on the critical line $\sigma = \frac{1}{2}$ by the Riemann Hypothesis (RH). While RH constrains the real parts of these zeros, their imaginary parts $\gamma$ display complex statistical behavior that is still not fully understood. Among the conjectures related to this behavior, the Linear Independence Conjecture (LIC), which asserts that the ordinates $\gamma$ are linearly independent over $\mathbb{Q}$, has deep implications for both the distribution of prime numbers and the analytic properties of $\zeta(s)$.

The conjecture was first introduced by Wintner~\cite{Wintner1935}, who used it to study the error term in the prime number theorem. It has since led to a variety of consequences, some of which were later proven unconditionally. For instance, Ingham used LIC to disprove Mertens' conjecture~\cite{Ingham1942}, which was eventually shown to be false without assuming LIC by Odlyzko and te Riele~\cite{OdlyzkoTeRiele1985}.

Montgomery's quantitative version of the Linear Independence Conjecture indicates that, assuming the Riemann Hypothesis and the linear independence of the ordinates of the nontrivial zeros, the normalized prime-counting error $f(u)$ in~\cite[Eq.~(13.17)]{MontgomeryVaughan2007} has a limiting distribution equal to that of the random variable
\[
X = 2 \sum_{\gamma > 0} \frac{\cos 2\pi X_{\gamma}}{|\rho|},
\]
where the $X_\gamma$ are independent random variables uniformly distributed on $[0,1]$. This reflects the heuristic that the ordinates of the zeros behave like independent random inputs. Moreover, the distribution function $F_X(x)$ satisfies the unconditional tail bounds
\[
\exp\left(-c_1\sqrt{x}e^{\sqrt{2\pi x}}\right) < 1 - F_X(x) < \exp\left(-c_2\sqrt{x}e^{\sqrt{2\pi x}}\right) \quad (x \geq 2),
\]
for some absolute constants $c_1, c_2 > 0$~\cite[Section~13.1]{MontgomeryVaughan2007}.

If LIC holds, it would imply, for example, that the shifts $\gamma + y$ (for fixed $y \neq 0$) never coincide with other zeros. A necessary weak form of this property—that zeros avoid fixed displacements—was recently established unconditionally under RH by Banks:

\begin{theorem}[{\cite[Theorem 1.1]{Banks2024}}]\label{thm:Banks}
Assume the Riemann Hypothesis. Let $y \neq 0$ and $C > 0$ be fixed. There exists $T_* = T_*(y,C) > 0$ such that for every $T \geq T_*$, the interval $[T, T(1 + \epsilon)]$ with $\epsilon = T^{-C / \log \log T}$ contains an ordinate $\gamma$ of a zero of $\zeta(s)$ satisfying
\[
\zeta(\tfrac{1}{2} + i\gamma) = 0 \quad \text{and} \quad \zeta(\tfrac{1}{2} + i(\gamma + y)) \neq 0.
\]
\end{theorem}

This result demonstrates that every sufficiently long interval contains a zero whose $y$-shift is not a zero, a necessary (though significantly weaker) condition than LIC. 

Note that in our setting we instead work with $\epsilon = \exp(-C\sqrt{\log T})$, which differs substantially from the choice made by Banks; see the remark following Lemma~\ref{lem:A2} for further discussion.

A natural extension of Banks’ result is the following problem.  
Let $y_1, \dots, y_k$ be distinct nonzero real numbers. One would like to prove - assuming RH - that for sufficiently large $T$, there exists a zero $ \tfrac{1}{2} + i\gamma $ with
\[
\zeta\!\left(\tfrac{1}{2} + i\gamma\right) = 0 \quad \text{and} \quad 
\zeta\!\left(\tfrac{1}{2} + i(\gamma + y_j)\right) \neq 0 \ \text{for all } 1 \leq j \leq k,
\]
where $\gamma$ lies in a short interval $[T, (1+\epsilon)T]$ (with $\epsilon \to 0$ slowly as $T \to \infty$).

Banks established this when $k=1$, by studying the nonvanishing of certain sums over zeros, specifically
\[
S_1(T_1, T_2) \;=\; 
\sum_{\substack{\rho = \tfrac{1}{2} + i\gamma \\ T_1 < \gamma < T_2}} 
x^{\rho} \, \zeta(\rho + i y),
\]
where the sum runs over nontrivial zeros of $\zeta(s)$.

Attempting to extend this directly suggests considering the multivariate analogue
\[
S_k(T_1, T_2) \;=\; 
\sum_{\substack{\rho = \tfrac{1}{2} + i\gamma \\ T_1 < \gamma < T_2}} 
x^{\rho} \, \zeta(\rho + i y_1)\cdots \zeta(\rho + i y_k).
\]
However, for $k \geqslant 2$ the combinatorial complexity of the Dirichlet coefficients and the associated contour integrals increases dramatically, making the method intractable.

In the special case $k=2$, one can partially reduce this difficulty by conjugating one factor. That is, instead of 
$\zeta(\rho + i y_1)\zeta(\rho + i y_2)$, one considers
\[
\sum_{\substack{\rho = \tfrac{1}{2} + i\gamma \\ T_1 < \gamma < T_2}} 
x^{\rho} \, \zeta(\rho + i y_1)\,\overline{\zeta(\rho + i y_2)}.
\]
Here the conjugated zeta factor is not analytic, but under RH one can replace it by $\zeta(1-s - i y_2)$ and analyze products of the form 
$\zeta(s+i y_1)\zeta(1-s-i y_2)$ using the functional equation. This makes progress possible when $k=2$.

For $k>2$, the situation appears to require a deeper framework. Indeed, conjectural asymptotics for averages of such products fall within the scope of the 
\emph{Ratios Conjecture}~\cite{CFZ2008,ConreySnaith2007}, which predicts correlations between multiple shifted zeta values. 
Proving results along these lines, even under RH, remains a challenging open problem.

It is expected that Banks's result holds for any finite number of shifts; however, such a generalization is well beyond our current reach. Even the extension to two shifts introduces significant challenges and complexity to the analysis, but it has been achieved and will be presented in this work. In fact, we show that for sufficiently large $T$, the zeros of the Riemann zeta function, when shifted by these two values, do not coincide with other zeros within certain intervals. This represents the first theorem addressing the simultaneous avoidance of zeros under multiple fixed shifts and brings us a step closer to understanding the kind of independence suggested by LIC.

Throughout this paper, the assumption of the Riemann Hypothesis (RH) will be explicitly stated in the hypothesis of any theorem, lemma, or corollary where it is required. Results that are unconditional will be presented as such. This convention is adopted to provide maximum clarity and to facilitate the potential future removal of the RH assumption from conditional results.

\subsection{Notations}
Throughout this paper, any implied constants in the symbols \( O \), \( \ll \), and \( \gg \) are absolute unless specified otherwise. We also will use the notation \(\mathbf{e}(u) := \exp(2 \pi i u)\) for all \(u \in \mathbb{R}\). Additionally, we will assume that \( T_1 \), \( T_2 \) satisfy the condition

\[
2 T_1 > T_2 > T_1 > T_0,
\]
and will define the following quantities

\[
\Delta := T_2 - T_1, \quad T := \frac{1}{2}(T_1 + T_2), \quad L := \exp\left( \frac{\log T}{\log \log T} \right).
\]

\subsection{Main Result}

\begin{theorem}
\label{thm:1.1}
For any fixed real \( y_1, y_2 \neq 0 \) and \( C > 0 \), there exists a constant \( T_{*} = T_{*}(y_1, y_2, C) > 0 \) such that for every \( T \geqslant T_{*} \), the following conditions hold for at least one \( \gamma \) in the interval \( [T, T(1 + \epsilon)] \), where $\epsilon := \exp\left(-C \sqrt{\log T}\right)$

\[
\zeta\left( \frac{1}{2} + i \gamma \right) = 0, \quad \zeta\left( \frac{1}{2} + i(\gamma + y_1) \right) \neq 0, \quad \zeta\left( \frac{1}{2} + i(\gamma + y_2) \right) \neq 0.
\]
    
\end{theorem}

\section{outline of our argument}

To prove Theorem \ref{thm:1.1}, we study sums of the form
\begin{equation}
\sum_{\substack{\rho = \frac{1}{2} + i \gamma \\ T_1 < \gamma < T_2}} x^{\rho} \zeta(\rho + i y_1) \overline{\zeta(\rho + i y_2)}, \tag{2.1} \label{eq:2.1}
\end{equation}
where $T_1$ is large and $T_1 \asymp T_2$. To estimate such sums, we apply the residue theorem to the meromorphic function
\begin{equation}
\mathcal{D}_y(s) := \frac{\zeta'(s)}{\zeta(s)} \zeta(s + i y_1) \zeta(1 - s - i y_2). \tag{2.2} \label{eq:2.2}
\end{equation}

Note that $\mathcal{D}_y(s)$ cannot be represented as an absolutely convergent Dirichlet series in the half-plane $\sigma > 1$. Therefore, we rely on the functional equation of the zeta function.

To evaluate the sum in \eqref{eq:2.1}, we use Cauchy’s residue theorem and deform the integration paths to relate the zeros of the zeta function in the sum to its singularities. Additionally, by employing the reflection property across the critical line, as implied by the functional equation and under the assumption of the Riemann Hypothesis, and also applying Schwarz’s reflection principle, we can replace $\overline{\zeta(s + i y_2)}$ with $\zeta(1 - s - i y_2)$. This step is necessary since the conjugate of an analytic function is not necessarily analytic.

Thus, the sum in \eqref{eq:2.1} can be expressed as a sum of four integrals, $I_1, \dots, I_4$, corresponding to the four segments of a rectangular contour enclosing the zeros of the zeta function appearing in the sum. The integrals along the right and left vertical edges yield the main contribution, while the horizontal integrals contribute negligibly.

To evaluate the right vertical integral, we expand the product of zeta functions in the integrand into a Dirichlet series. For the left vertical integral, we use the functional equation to shift the line of integration to the right of the 1-line, allowing us to apply the Dirichlet series. After this transformation, we apply Perron’s formula and related analytic techniques to convert these sums into manageable integrals.

Among the resulting expressions, we isolate a main term. We show that this main term does not vanish by demonstrating that it includes a coefficient bounded below by a positive constant, independent of the specific parameters involved. The remaining terms are shown to be small in comparison, through a combination of analytic estimates and known bounds on the zeta function and related sums.

More technically, we aim to show that
\[
\sum_{\substack{\rho = \frac{1}{2} + i \gamma \\ T_1 < \gamma < T_2}} x^{\rho} \zeta(\rho + i y_1) \overline{\zeta(\rho + i y_2)} \neq 0
\]
by proving the following

\begin{theorem}
\label{thm:2.1}
Assume RH. Let $y_1, y_2 \neq 0$ and $C > 0$ be constants. Then there exists an absolute positive constant $A$ and a function $M_1=M_1(y_1, y_2, x)$ such that $M_1 > A$, and also a function $T_* = T_*(y_1, y_2, C) > 0$ such that for all $T_2 > T_1 \geq T_*$, the following holds:

For $\Delta := T_2 - T_1$ and $x$ a prime in the range $x \asymp \log T$, we have
\begin{equation}
\sum_{\substack{\rho = \frac{1}{2} + i\gamma \\ T_1 < \gamma < T_2}} x^\rho \zeta(\rho + i y_1) \overline{\zeta(\rho + i y_2)} = \frac{\Delta \log T}{2\pi} \cdot x^{-i y_1} \zeta(1 - i(y_2 - y_1)) M_1 + O\left(\Delta \log^{1/2} T \right), \tag{2.3} \label{eq:2.3}
\end{equation}
where $T = \frac{T_1 + T_2}{2}$ and the implied constant in the error term depends only on $y_1$, $y_2$, and $C$. Here, the sum runs over the zeros of $\zeta(s)$, each counted with multiplicity.
\end{theorem}

In Section~\ref{sec:4.5}, we show that Theorem~\ref{thm:2.1} implies Theorem~\ref{thm:1.1}.

To prove this theorem, we present several key lemmas in Section~\ref{sec:3} that assist in the analysis. These lemmas are essential for understanding the properties of the zeta and L-functions and their logarithmic derivatives, especially in the context of integrals involving shifts in the complex plane. Additionally, we introduce several lemmas that provide bounds on various sums, which are crucial for the subsequent calculations. Finally, in Chapter 4, we complete the proof of Theorem~\ref{thm:1.1}.

\section{Preliminaries} \label{sec:3}

\subsection{The Zeta Function}

\begin{lemma}[Functional Equations] \label{lem:3.1}
The Riemann zeta function satisfies the following identities:

\[
\zeta(s) = \mathcal{X}(s)\, \zeta(1 - s), \tag{3.1} \label{eq:3.1}
\]
and
\[
\frac{\zeta'}{\zeta}(s) = -\frac{\zeta'}{\zeta}(1 - s) + \log \pi - \frac{1}{2} \psi\left( \frac{s}{2} \right) - \frac{1}{2} \psi\left( \frac{1 - s}{2} \right), \tag{3.2} \label{eq:3.2}
\]
where
\[
\mathcal{X}(s) := 2^s \pi^{s - 1} \Gamma(1 - s) \sin\left( \frac{\pi s}{2} \right), \quad \text{and} \quad \psi(s) := \frac{\Gamma'(s)}{\Gamma(s)}.\tag{3.3} \label{eq:3.3}
\]

\end{lemma}

\begin{proof}

\cite[Eq. 2.1.1]{Titchmarsh1986}.

\end{proof}

\begin{lemma}
\label{lem:3.2}

If \( \rho = \beta + i\gamma \) runs through the zeros of \( \zeta(s) \), then
\[
\frac{\zeta'(s)}{\zeta(s)} = \sum_{\substack{\rho = \beta + i\gamma \\ |\gamma - t| \le 1}} \frac{1}{s - \rho} + O(\log |t|), \tag{3.4} \label{eq:3.4}
\]
uniformly for \( -1 \leqslant \sigma \leqslant 2 \).

\end{lemma}

\begin{proof}

\cite[Theorem 9.6 (A)] {Titchmarsh1986}.

\end{proof}

\begin{lemma}
\label{lem:3.3}

Assume RH. There exists an absolute constant \( \lambda > 0 \) such that

\[
|\zeta(s)| \leqslant \exp \left( \frac{\lambda \log |t|}{\log \log |t|} \right) \tag{3.5} \label{eq:3.5}
\]
holds uniformly for \( \sigma \geqslant \frac{1}{2} - \frac{1}{\log \log |t|} \) and \( |t| \geqslant 10 \).

\end{lemma}

\begin{proof}

\cite[Lemma 2.6]{Banks2024}.

\end{proof}

\begin{lemma}
\label{lem:3.4}

We have

\[
\zeta(s) = O(\log t) \tag{3.6} \label{eq:3.6}
\]
uniformly in the region
\[
\sigma > 1 - \frac{c}{\log |t|}, \quad \text{and} \quad |t| > e,
\]
where $c$ is any positive constant. 

\end{lemma}

\begin{proof}

\cite[Theorem 3.5] {Titchmarsh1986}.

\end{proof}

\begin{lemma}
\label{lem:3.5}

Let \( \delta > 0 \) be fixed. Then

\[
\zeta(s) = \frac{1}{s - 1} + O(1)
\]
uniformly for \( s \) in the rectangle \( \delta \leq \sigma \leq 2 \) and \( |t| \ll 1 \); and

\[
\zeta(s) \ll (1 + |t|^{1 - \sigma}) \min \left\{ \frac{1}{|\sigma - 1|}, \log |t| \right\} \tag{3.7} \label{eq:3.7}
\]
uniformly for \( \delta \leq \sigma \leq 2 \) and \( |t| \geq 5 \). 

\end{lemma}

\begin{proof}

\cite[Corollary 1.17]{MontgomeryVaughan2007}.

\end{proof}

\begin{lemma}
\label{lem:3.6}

Assume RH. Then

\[
\frac{\zeta'}{\zeta}(s) \ll \left((\log |t|)^{2 - 2\sigma} + 1\right) \min \left\{ \frac{1}{|\sigma - 1|}, \log \log |t| \right\} \tag{3.8} \label{eq:3.8}
\]
uniformly for \( \frac{1}{2} + \frac{1}{\log \log |t|} \leq \sigma \leq \frac{3}{2} \) and \( |t| \geq 5 \).
\end{lemma}

\begin{proof}

\cite[Corollary 13.14]{MontgomeryVaughan2007}.

\end{proof}

\subsection{Estimates with \( \mathcal{X}(s) \) }

\begin{lemma}
\label{lem:3.7}

Let \( \mathcal{I} \subset \mathbb{R} \) be a bounded interval. Uniformly for \( \sigma \in \mathcal{I} \) and \( t \geq 1 \), we have the asymptotic estimate

\[
\mathcal{X}(1 - \sigma + i t) = \mathrm{e}^{\pi i / 4} \exp\left(-i t \log \left(\frac{t}{2 \pi \mathrm{e}}\right)\right) \left(\frac{t}{2 \pi}\right)^{\sigma - 1 / 2} \left\{ 1 + O_{\mathcal{I}}\left( t^{-1} \right) \right\}.
\]

\end{lemma}

\begin{proof}

\cite[Lemma~2.7]{Banks2024}.

\end{proof}

\begin{lemma}
\label{lem:3.8}

Uniformly for \( 10a \geqslant b > a \geqslant 10 \), \( \sigma \in \left[\frac{1}{10}, 10\right] \), and \( m \in \{0,1\} \), we have

\[
\begin{aligned}
& \int_{a}^{b} \exp\left( i t \log \left( \frac{t}{u \cdot \mathrm{e}} \right) \right) \left( \frac{t}{2 \pi} \right)^{\sigma - 1/2} \left( \log \frac{t}{2 \pi} \right)^m \, dt \\
& \quad = (2 \pi)^{1-\sigma} u^{\sigma} \exp\left( -\left( u - \frac{\pi}{4} \right) i \right) \left( \log \frac{u}{2 \pi} \right)^m \cdot \mathbf{1}(a, b; u) + O\left( E \left( \log a \right)^m \right) 
\end{aligned}
\]
where

\[
\mathbf{1}(a, b; u) := 
\begin{cases}
1 & \text{if } a < u \leqslant b \\
0 & \text{otherwise}
\end{cases}
\]
and

\[
E := a^{\sigma - 1/2} + \frac{a^{\sigma + 1/2}}{|a - u| + a^{1/2}} + \frac{b^{\sigma + 1/2}}{|b - u| + b^{1/2}}.
\]

\end{lemma}

\begin{proof}

\cite[Lemmas 2, 3]{Gonek1984}.

\end{proof}

\begin{lemma}
\label{lem:3.9}

Let \( f_j(s) \) and \( g_j(v) \) be defined as follows

\[
f_j(s) := \begin{cases}
1 & \text{if } j = 0, \\
\psi\left( \frac{s}{2} \right) & \text{if } j = +1, \\
\psi\left( \frac{1 - s}{2} \right) & \text{if } j = -1,
\end{cases}
\]
and

\[
g_j(v) := \begin{cases}
1 & \text{if } j = 0, \\
\log(\pi v) - \frac{\pi i}{2} & \text{if } j = +1, \\
\log(\pi v) + \frac{\pi i}{2} & \text{if } j = -1.
\end{cases}
\]

For each \( j \in \{ 0, \pm 1 \} \), and uniformly for \( 2T_1 > T_2 > T_1 \geq 10 (|y| + 1) \), \( c \in [1/2, 10] \), and \( v > 0 \), we have the uniform estimate

\[
\frac{1}{2 \pi i} \int_{c - i T_2}^{c - i T_1} \mathcal{X}(1 - s + i y) f_j(s) v^{-s} \, ds
= v^{-i y} \mathbf{e}(v) g_j(v) \cdot \mathbf{1}\left( T_1', T_2'; 2 \pi v \right) + O_y(E), \tag{3.10} \label{eq:3.10}
\]
where \( T_j' := T_j + y \) for \( j \in \{ 1, 2 \} \), and

\[
E := \frac{\left( \log T_1' \right)^{|j|}}{v^c} \left\{ \left( T_1' \right)^{c - 1/2} + \frac{\left( T_1' \right)^{c + 1/2}}{\left| T_1' - 2 \pi v \right| + \left( T_1' \right)^{1/2}} + \frac{\left( T_2' \right)^{c + 1/2}}{\left| T_2' - 2 \pi v \right| + \left( T_2' \right)^{1/2}} \right\}. \tag{3.11} \label{eq:3.11}
\]

The implied constant depends only on \( y \).

\end{lemma}

\begin{proof}

\cite[Lemma 2.9]{Banks2024}.

\end{proof}

\subsection{Dirichlet L-Functions}

\begin{lemma}
\label{lem:3.10}  

Define \( \tau = |t| + 4 \). There is an absolute constant \( c > 0 \) such that if \( \chi \) is a Dirichlet character modulo \( x \), then the region

\[
R_x = \left\{ s : \sigma > 1 - \frac{c}{\log x \tau} \right\} \tag{3.10} \label{eq:3.12}
\]
contains no zero of \( L(s, \chi) \) unless \( \chi \) is a quadratic character, in which case \( L(s, \chi) \) has at most one, necessarily real, zero \( \beta < 1 \) in \( R_x \).

\end{lemma}

\begin{proof}
  
\cite[Theorem 11.3]{MontgomeryVaughan2007}.
  
\end{proof}

\begin{lemma}
\label{lem:3.11}

There is an absolute constant \( c > 0 \) such that if \( \chi \) is a quadratic character modulo \( x \) and \( L(s, \chi) \) has an exceptional zero \( \beta \), then

\[
\beta \leq 1 - \frac{c}{x^{1/2} (\log x)^2}. \tag{3.13} \label{eq:3.13}
\]

\end{lemma}

\begin{proof}

\cite[Corollary 11.12]{MontgomeryVaughan2007}.
    
\end{proof}

\begin{lemma}
\label{lem:3.12}
    
Let \( \chi \) be a non-principal character modulo \( x \), and let \( \delta > 0 \) be fixed. Then, for \( \delta \leq \sigma \leq 2 \), we have the bound

\[
L(s, \chi) \ll \left( 1 + (x \tau)^{1-\sigma} \right) \min \left( \frac{1}{|\sigma - 1|}, \log x \tau \right), \tag{3.14} \label{eq:3.14}
\]
uniformly.

\end{lemma} 

\begin{proof}

\cite[Lemma 10.15]{MontgomeryVaughan2007}.
    
\end{proof}

\begin{lemma}
\label{lem:3.13}

Let \( \chi \) be a non-principal character modulo \( x \), and let \( c \) be the constant in Lemma \ref{lem:3.10}. Suppose that \( \sigma \geq 1 - \frac{c}{2 \log x \tau} \). If \( L(s, \chi) \) has no exceptional zero, or if \( \beta_1 \) is an exceptional zero of \( L(s, \chi) \) and \( |s - \beta_1| \geq \frac{1}{\log x} \), then

\[
\frac{L'(s, \chi)}{L(s, \chi)} \ll \log x \tau. \tag{3.15} \label{eq:3.15}
\]

\end{lemma}

\begin{proof}

\cite[Lemma 11.14]{MontgomeryVaughan2007}.
    
\end{proof}

\subsection{Bounds on Various Sums}

\begin{lemma}
\label{lem:3.14}

We have the estimate:

\[
\left| \sum_{abc=n} \Lambda(a) a^{i\theta_1} b^{i\theta_2} c^{i\theta_3} \right| \ll d(n) \log n, \tag{3.16} \label{eq:3.16}
\]
where \( \theta_i \in \mathbb{R} \).

\end{lemma}

\begin{proof}
    
Writing \( n = p_1^{a_1} p_2^{a_2} \cdots p_r^{a_r} \), we estimate
\[
\left| \sum_{abc=n} \Lambda(a) a^{i\theta_1} b^{i\theta_2} c^{i\theta_3} \right| 
\leq \sum_{abc=n} \Lambda(a).
\]

Since \( \Lambda(a) \) is nonzero only when \( a = p^k \), we have
\[
\sum_{abc=n} \Lambda(a) = \sum_{a \mid n} \Lambda(a) \cdot \#\{(b, c) : bc = n/a\}.
\]

For each such \( a \), the number of solutions to \( bc = n/a \) is \( d(n/a) \), so
\[
\sum_{abc=n} \Lambda(a) \leq \sum_{a \mid n} \Lambda(a) d(n/a).
\]

This is bounded above by
\[
\left( \max_{a \mid n} d(n/a) \right) \sum_{a \mid n} \Lambda(a) 
\leq d(n) \log n,
\]
since \( \sum_{a \mid n} \Lambda(a) = \log n \).

\end{proof}

\begin{lemma}
\label{lem:3.15}

Let $X > 1$ and let $Y$ be a real number such that $0 < Y \leq cX$ for some constant $c < 1$ (in particular, $Y = o(X)$). Then we have the estimate
\[
(X + Y) \log(X + Y) - (X - Y) \log(X - Y) \ll Y \log X.
\tag{3.17} \label{eq:3.17}
\]
\end{lemma}

\begin{proof}
Consider the function $f(y) = (X + y) \log(X + y)$ defined for $y > -X$. This function is differentiable on $(-X, \infty)$. Its derivative is
\[
f'(y) = \log(X + y) + \frac{X + y}{X + y} = \log(X + y) + 1.
\]

Applying the Mean Value Theorem on the interval $[-Y, Y]$, we find that there exists some $\xi$ with $|\xi| < Y$ such that:
\[
f(Y) - f(-Y) = (Y - (-Y)) \cdot f'(\xi) = 2Y \cdot (\log(X + \xi) + 1).
\]

Since $|\xi| < Y \leq cX$, we have $X - Y \leq X + \xi \leq X + Y \ll X$. Therefore, $\log(X + \xi) \leq \log(X + Y) \ll \log X$. It follows that
\[
(X + Y) \log(X + Y) - (X - Y) \log(X - Y) = f(Y) - f(-Y) = 2Y (\log(X+\xi) + 1) \ll Y \log X,
\]
which completes the proof.
\end{proof}

\begin{lemma}
\label{lem:3.16}

For any \( 1 < A < B \), we have the estimate
\[
\sum_{A < n \leq B} \frac{d(n)}{n} \ll \log^2 B. \tag{3.18} \label{eq:3.18}
\]
\end{lemma}

\begin{proof}
We give two proofs. 

First, let $\delta = 1/\log B$. For $n \leq B$, we have $n^\delta \leq B^\delta = e$, so
\[
\frac{d(n)}{n} 
= \frac{d(n)}{n^{1+\delta}} \cdot n^\delta 
\ll \frac{d(n)}{n^{1+\delta}}.
\]
Hence
\[
\sum_{A < n \leq B} \frac{d(n)}{n} 
\ll \sum_{n \leq B} \frac{d(n)}{n} 
\ll \sum_{n=1}^\infty \frac{d(n)}{n^{1+\delta}} 
= \zeta(1+\delta)^2,
\]
where the last equality uses the Dirichlet series identity 
\(\sum_{n=1}^\infty d(n)n^{-s} = \zeta(s)^2\) for $\Re(s) > 1$. 
Since $\zeta(1+\delta) \sim 1/\delta$ as $\delta \to 0^+$, it follows that 
\[
\zeta(1+\delta)^2 \ll \delta^{-2} = (\log B)^2.
\]

For later purposes we also require a more precise asymptotic. 
Applying Abel’s summation formula (see \cite[Theorem~4.2]{Apostol2013}) together with 
Voronoï’s bound for the error term in the divisor problem \cite{Voronoi1904}, we obtain
\[
\sum_{A < n \leq B} \frac{d(n)}{n}
= B^{-1} \Big( B \log B + (2\gamma - 1)B + O(B^{1/3}) \Big)
- A^{-1} \Big( A \log A + (2\gamma - 1)A + O(A^{1/3}) \Big)
\]
\[
\quad + \int_A^B t^{-2} \Big( t \log t + (2\gamma - 1)t + O(t^{1/3}) \Big)\, dt.
\]
Simplifying gives
\[
\sum_{A < n \leq B} \frac{d(n)}{n}
= 2\gamma \log \frac{B}{A}
+ \tfrac{1}{2} \big( \log^2 B - \log^2 A \big)
+ O\!\left( B^{-2/3} + A^{-2/3} \right),
\]
which again shows the bound \(\ll \log^2 B\).
\end{proof}

\begin{lemma}\label{lem:3.17}
Let $X$ and $Y$ be as in Lemma~\ref{lem:3.15}. In particular, $Y \ll X$. Then we have the estimates
\[
S_1 := \sum_{X/2 < n < X - Y} \frac{d(n)}{n (X - n + Y)} \ll \frac{\log^2 X}{X},
\]
and
\[
S_2 := \sum_{X + Y < n < 3X} \frac{d(n)}{n (n - X + Y)} \ll \frac{\log^2 X}{X}.
\]
\end{lemma}

\begin{proof}
We begin with first sum. By Abel's summation formula we have

$$
S_1 = S(X - Y) f(X - Y) - S(X/2) f(X/2) - \int_{X/2}^{X - Y} S(u) f'(u) \, du,
$$

where $S(u) = \sum_{X/2 < n \le u} \frac{d(n)}{n}$ and $f(u) = \frac{1}{X - u + Y}$.

By Lemma \ref{lem:3.16} we have the approximation

$$
S(u) = \frac{1}{2} \log^2 u + 2\gamma \log u + C_X + O(X^{-2/3}),
$$

where $$C_X = -\frac{1}{2} \log^2 X/2 - 2\gamma \log X/2.$$

Therefore, the first boundary term becomes

$$
S(X - Y) f(X - Y) = \left( \frac{1}{2} \log^2(X - Y) + 2\gamma \log(X - Y) + C_X + O(X^{-2/3}) \right) \cdot \frac{1}{2Y}.
$$

Using the approximation $X - Y = X \left(1 - \frac{Y}{X} \right)$, and assuming $Y \ll X$, we expand

$$
\log(X - Y) = \log X + \log\left(1 - \frac{Y}{X} \right) = \log X - \frac{Y}{X} + O\left(\left(\frac{Y}{X}\right)^2\right).
$$

Hence,

$$
\log^2(X - Y) = \log^2 X - 2 \log X \cdot \frac{Y}{X} + O\left(\frac{Y^2}{X^2} \log X\right).
$$

So,

$$
S(X - Y) f(X - Y) = \left( \frac{1}{2} \log^2 X + 2\gamma \log X + C_X + O\left(\frac{Y}{X} \log X\right) \right) \cdot \frac{1}{2Y}.
$$

For the second boundary term we have

$$
S(X/2) f(X/2) = \left( \frac{1}{2} \log^2(X/2) + 2\gamma \log(X/2) + C_X + O(X^{-2/3}) \right) \cdot \frac{1}{X/2 + Y}.
$$

Here, $
f(X/2) = \frac{1}{X/2 + Y} = \frac{2}{X} + O\left( \frac{Y}{X^2} \right),
$ since $Y \ll X$. Also, note that $C_X$ cancels the first two terms. Therefore, because $
S(X/2) = O(X^{-2/3})
$, we ultimately have

$$
S(X/2) f(X/2) = O(X^{-5/3}).
$$

We break the integral as

$$
I = \int_{X/2}^{X - Y} \left( \frac{1}{2} \log^2 u + 2\gamma \log u + C_X + O(X^{-2/3}) \right) \cdot \frac{1}{(X - u + Y)^2} \, du,
$$

$$
= I_1 + I_2 + I_3 + I_4.
$$

By a change of variable, $w = X - u + Y \Rightarrow u = X + Y - w$,
when $u = X/2 \Rightarrow w = X/2 + Y$,
and when $u = X - Y \Rightarrow w = 2Y$.

So the first integral becomes

$$
I_1 = \frac{1}{2} \int_{2Y}^{X/2 + Y} \log^2(X + Y - w) \cdot \frac{1}{w^2} \, dw.
$$

For $w \in [2Y, X/2 + Y]$, we have $X + Y - w \in [X/2, X - Y]$. Assuming $Y \ll X$, we write
\[
\log^2(X + Y - w) = \log^2 X + O\left( \frac{|Y - w|}{X} \log X \right) + O\left( \frac{(Y - w)^2}{X^2} \right).
\]

Substituting this into the integral, we obtain
\[
I_1 = \frac{1}{2} \int_{2Y}^{X/2 + Y} \left[ \log^2 X + O\left( \frac{|Y - w|}{X} \log X \right) + O\left( \frac{(Y - w)^2}{X^2} \right) \right] \cdot \frac{1}{w^2} \, dw.
\]

Breaking this into three parts, we get
\[
I_1 = \frac{\log^2 X}{2} \int_{2Y}^{X/2 + Y} \frac{1}{w^2} \, dw + O\left( \frac{\log X}{X} \int_{2Y}^{X/2 + Y} \frac{|Y - w|}{w^2} \, dw \right) + O\left( \frac{1}{X^2} \int_{2Y}^{X/2 + Y} \frac{(Y - w)^2}{w^2} \, dw \right).
\]

\[
\int_{2Y}^{X/2 + Y} \frac{1}{w^2} \, dw = \left[ -\frac{1}{w} \right]_{2Y}^{X/2 + Y} = \frac{1}{2Y} - \frac{1}{X/2 + Y}.
\]

So the main term is
\[
\frac{\log^2 X}{2} \left( \frac{1}{2Y} - \frac{1}{X/2 + Y} \right).
\]

For the error integrals, $w \geq 2Y > Y$, and thus $|Y - w| = w - Y$. Therefore, the first error integral becomes  
\[
\frac{\log X}{X} \int_{2Y}^{X/2 + Y} \frac{w - Y}{w^2} \, dw = \frac{\log X}{X} \left( \log\left( \frac{X + 2Y}{4Y} \right) - \left( \frac{1}{2} - \frac{Y}{X/2 + Y} \right) \right),
\]

which equals
\[
\frac{\log X}{X} \left( \log\left(\frac{X}{4Y}\right) - \frac{1}{2} + \frac{4Y}{X} + O\left(\frac{Y^2}{X^2}\right) \right) \ll \frac{\log^2 X}{X},
\]
for $Y \ll X$.

Now we estimate the second error integral. We expand $(Y - w)^2 = (w - Y)^2 = w^2 - 2Yw + Y^2$, then

\[
\int_{2Y}^{X/2 + Y} \frac{(w - Y)^2}{w^2} \, dw = \int_{2Y}^{X/2 + Y} \left( 1 - \frac{2Y}{w} + \frac{Y^2}{w^2} \right) dw,
\]

which equals
\[
\left( \frac{X}{2} - Y \right) - 2Y \log\left( \frac{X/2 + Y}{2Y} \right) + Y^2 \left( \frac{1}{2Y} - \frac{1}{X/2 + Y} \right),
\]
which is $O(X)$ if $Y \ll X$.

So the second error term is
\[
O\left( \frac{X}{X^2} \right) = O\left( \frac{1}{X} \right).
\]

Putting all terms together, we obtain

\[
I_1 = \frac{\log^2 X}{2} \left( \frac{1}{2Y} - \frac{1}{X/2 + Y} \right) + O\left( \frac{\log^2 X}{X} \right).
\]

Similarly, we have
\[
I_2 = 2 \gamma \Bigg(
\log X \left( \frac{1}{2Y} - \frac{2}{X} + O\left(\frac{Y}{X^2}\right) \right) 
+ \frac{1}{X} \left[ \frac{1}{2} - \frac{2Y}{X} - \log\left(\frac{X}{4Y}\right) + O\left(\frac{Y}{X}\right) \right] 
\Bigg)
\]
\[
\quad
- 2 \gamma \Bigg(
\frac{1}{2 X^2} \left[ \frac{X}{2} - Y - 2Y \log\left(\frac{X}{4Y}\right) \right] 
- O\left( \frac{1}{X} \right)
\Bigg)
= \frac{\gamma \log X}{Y} + O\left(\frac{\log X}{X}\right).
\]

For $I_3$ we have 

$$
I_3 = C_X \left( \frac{1}{Y} - \frac{1}{X/2 + Y} \right) 
$$

$$
= \frac{ -\tfrac{1}{2} \log^2 X - 2\gamma \log X + \log X \log 2 + 2\gamma \log 2 - \tfrac{1}{2} \log^2 2 }{2Y} + O\left( \frac{\log^2 X}{X} \right)
$$
And finally, for $I_4$, we obtain

$$
I_4 \ll X^{-2/3} \cdot \int_Y^{X} \frac{1}{w^2} \, dw \ll \frac{X^{-2/3}}{Y}.
$$

Combining all of these results, and noting that many main terms cancel, we also take into account that $Y \ll X$, we ultimately obtain

$$
S_1 \ll \frac{\log^2 X}{X}.
$$

By similar calculation we obtain 
\[
S_2 \ll \frac{\log^2 X}{X}.
\]
    
\end{proof}

\begin{lemma}
\label{lem:3.18}

Let \( t \geqslant 10x \geqslant 100 \), and define \( c := 1 + \frac{1}{\log x} \). Then, we have the estimate

\[
\sum_{n \geqslant 2} \frac{d(n) \log n}{n^c \left( |t - 2 \pi n x| + t^{1/2} \right)} \ll \frac{\log^3 t + \log^3 x}{t}. \tag{3.19} \label{eq:3.19}
\]

\end{lemma}

\begin{proof}

Following Banks~\cite[Lemma~2.4]{Banks2024}, we observe that the sum on the left-hand side is given by

\[
\left( \sum_{n \in S_1} + \sum_{n \in S_2} + \sum_{n \in S_3} \right) \frac{d(n) \log n}{2 \pi x n^c \left( |n - X| + Y \right)},
\]
where

\[
X := \frac{t}{2 \pi x}, \quad Y := \frac{t^{1/2}}{2 \pi x}.
\]

We define the sets as follows.

\[
\begin{aligned}
S_1 & := \{ n \geq 2 : |n - X| \leq Y \}, \\
S_2 & := \{ n \geq 2 : Y < |n - X| \leq 2X \}, \\
S_3 & := \{ n \geq 2 : |n - X| > 2X \}.
\end{aligned}
\]

For each \( n \in S_1 \), we have \( n \asymp X \), so the sum over \( n \in S_1 \) is

\[
\ll \frac{\log X}{x X^c Y} \sum_{X - Y \leq n \leq X + Y} d(n).
\]

Using Voronoi's upper bound for the Dirichlet divisor problem, which states

\[
\sum_{n \leq x} d(n) = x \log x + (2\gamma - 1)x + O(x^{1/3}),
\]
and applying Lemma \ref{lem:3.15}, we obtain

\[
\sum_{X - Y \leq n \leq X + Y} d(n) \ll Y \log X.
\]

Note that the error term is asymptotically negligible compared to the main term under the constraint \( x = O(t^{1/4}) \).

Thus, the first sum is

\[
\ll \frac{\log^2 t}{t^c}.
\]

Since \( n \ll X \) and \( n^{-c} \ll n^{-1} \) for each \( n \in S_2 \), the sum over \( n \in S_2 \) is

\[
\ll \frac{\log X}{x} \left( \sum_{2 \leq n < X - Y} + \sum_{X + Y < n \leq 3X} \right) \frac{d(n)}{n(|n - X| + Y)} := \frac{\log X}{x} (E_1 + E_2).
\]

To estimate the sum $E_1$ we split the sum into two intervals and we write $E_1 = E_1^{(1)}+E_1^{(2)}$, where

\[
E_1^{(1)} = \sum_{2 \leq n \leq X/2} \frac{d(n)}{n(X - n + Y)} \quad 
 \text{and} \quad 
E_1^{(2)} = \sum_{X/2 < n < X - Y} \frac{d(n)}{n(X - n + Y)}.
\]

For $n \leq X/2$, we have $X-n+Y \geq X/2$, so
\[
E_1^{(1)} \leq \frac{2}{X} \sum_{n \leq X/2} \frac{d(n)}{n} \ll \frac{\log^2 X}{X} \ll \frac{x \log^2 t}{t}, 
\]
where we applied Lemma \ref{lem:3.16}.

$E_1^{(2)}$ and $E_2$, have been done in Lemma~\ref{lem:3.17}. 
Therefore, the total sum over \(S_2\) satisfies
\[
\sum_{n \in S_2} \frac{d(n) \log n}{n^c (|n - X| + Y)} \ll \frac{\log^3 t}{t}.
\]

Finally, since \( n > 3X \) for each \( n \in S_3 \), we have \( (|n - X| + Y)^{-1} < (2X+Y)^{-1} \ll X^{-1} \). Thus, the sum over \( n \in S_3 \) is

\[
\ll x^{-1} X^{-1} \sum_{n > 3X} \frac{d(n) \log n}{n^{c}} \ll -t^{-1} \zeta'(c) \zeta(c) \ll \frac{\log^3 x}{t}.
\]

Putting everything together, we obtain the stated bound.

\end{proof}

\begin{lemma}
\label{lem:3.19}

For \( t, x \geq 3 \), and define \( c := 1 + \frac{1}{\log x} \). Then, we have the estimate

\[
\sum_{n \geq 2} \frac{d(n) \log n}{n^{c} \left( |t - \frac{2 \pi n}{x}| + t^{1/2} \right)} \ll \frac{\log^3 t + \log^3 x}{t}. \tag{3.20} \label{eq:3.20}
\]

\end{lemma}

\begin{proof}

This is similar to Lemma \ref{lem:3.18}, so we omit the repetition.

\end{proof}

\section{proof of Theorem 1.1}
\normalfont

\subsection{Preliminaries}

To evaluate the sum in \eqref{eq:2.1}, we begin by analyzing the structure of the integrand and the location of the poles using complex analytic methods. In particular, we utilize Cauchy's Residue Theorem, the properties of the Riemann zeta function, and a suitable choice of contour to relate the sum over the non-trivial zeros of the zeta function to integrals over vertical and horizontal segments.

We first encapsulate the essential reduction as

\begin{lemma}
Let \( x > 1 \), \( y_1, y_2 \in \mathbb{R} \), and \( T_1 < T_2 \) be large. Let \( \rho = \tfrac{1}{2} + i \gamma \) denote the non-trivial zeros of the Riemann zeta function. Then we have the identity
\[
\sum_{\substack{\rho = \frac{1}{2} + i \gamma \\ T_1 < \gamma < T_2}} m_\rho \, x^{\rho} \zeta(\rho + i y_1) \overline{\zeta(\rho + i y_2)} 
= I_1 + I_2 + I_3 + I_4 - \Delta,
\]
where \(m_\rho\) is the multiplicity of the zero \(\rho\), and
\begin{align}
I_1 &= \frac{1}{2 \pi i} \int_{c + i T_1}^{c + i T_2} \frac{\zeta'(s)}{\zeta(s)} \zeta(s + i y_1) \zeta(1 - s - i y_2) x^s \, ds, \tag{4.1} \label{eq:4.1} \\
I_2 &= \frac{1}{2 \pi i} \int_{c + i T_2}^{b + i T_2} \frac{\zeta'(s)}{\zeta(s)} \zeta(s + i y_1) \zeta(1 - s - i y_2) x^s \, ds, \tag{4.2} \label{eq:4.2} \\
I_3 &= \frac{1}{2 \pi i} \int_{b + i T_2}^{b + i T_1} \frac{\zeta'(s)}{\zeta(s)} \zeta(s + i y_1) \zeta(1 - s - i y_2) x^s \, ds, \tag{4.3} \label{eq:4.3} \\
I_4 &= \frac{1}{2 \pi i} \int_{b + i T_1}^{c + i T_1} \frac{\zeta'(s)}{\zeta(s)} \zeta(s + i y_1) \zeta(1 - s - i y_2) x^s \, ds, \tag{4.4} \label{eq:4.4}
\end{align}
with
\[
c := 1 + \frac{1}{\log x}, 
\quad b := \frac{1}{2} - \frac{1}{\log \log T}, 
\quad T := \tfrac{1}{2}(T_1 + T_2),
\]
and
\[
\Delta := \operatorname{Res}_{s=1} g(s) 
+ \mathbf{1}_{\{1 - i y_1 \in \mathcal{C}\}} \operatorname{Res}_{s = 1 - i y_1} g(s) 
+ \mathbf{1}_{\{-i y_2 \in \mathcal{C}\}} \operatorname{Res}_{s = -i y_2} g(s),
\]
where
\[
g(s) := \frac{\zeta'(s)}{\zeta(s)} \zeta(s + i y_1) \zeta(1 - s - i y_2) x^s.
\]
\end{lemma}

\begin{proof}
Consider the meromorphic function
\[
g(s) = \frac{\zeta'(s)}{\zeta(s)} \zeta(s + i y_1) \zeta(1 - s - i y_2) x^s.
\]
We integrate \(g\) over the rectangular contour \(\mathcal{C}\) with vertical edges at \(\Re(s) = b, c\) and horizontal edges at \(\Im(s) = T_1, T_2\). We assume \(T_1, T_2\) are chosen so that no zero of \(\zeta(s)\) lies on the horizontal edges, and that none of the poles \(s=1\), \(s=1-i y_1\), \(s=-i y_2\) lies on \(\mathcal{C}\). If necessary, one may indent around boundary points; such contributions vanish as the indentation radius tends to zero.

The poles of \(g(s)\) inside \(\mathcal{C}\) are as follows.

(i) At non-trivial zeros \(\rho\) of \(\zeta(s)\), the pole of \(\zeta'(s)/\zeta(s)\) has residue \(m_\rho\), and hence
\[
\operatorname{Res}_{s=\rho} g(s) = m_\rho \, \zeta(\rho + i y_1) \zeta(1 - \rho - i y_2) x^\rho.
\]

(ii) At \(s=1\), we have
\[
\operatorname{Res}_{s=1} g(s) = - \zeta(1 + i y_1) \zeta(- i y_2) x.
\]

(iii) If \(1 - i y_1 \in \mathcal{C}\), then \(g\) has a simple pole there from \(\zeta(s+i y_1)\); the residue is
\[
\operatorname{Res}_{s=1 - i y_1} g(s) = \left. \frac{\zeta'(s)}{\zeta(s)} \zeta(1 - s - i y_2) x^s \right|_{s = 1 - i y_1},
\]
provided \(\zeta(s)\neq 0\) at that point. 

(iv) If \(- i y_2 \in \mathcal{C}\), then \(g\) has a simple pole there from \(\zeta(1 - s - i y_2)\); the residue is
\[
\operatorname{Res}_{s=- i y_2} g(s) = - \left. \frac{\zeta'(s)}{\zeta(s)} \zeta(s + i y_1) x^s \right|_{s = - i y_2},
\]
again provided \(\zeta(s)\neq 0\) there.

By Cauchy’s residue theorem we obtain
\[
\frac{1}{2\pi i} \int_{\mathcal{C}} g(s) \, ds
= \sum_{\substack{\rho \\ T_1 < \Im(\rho) < T_2}} m_\rho \, \zeta(\rho + i y_1) \zeta(1 - \rho - i y_2) x^\rho + \Delta.
\]
Writing the contour integral as a sum over its four sides yields \(I_1+I_2+I_3+I_4\), hence
\[
\sum_{\substack{\rho \\ T_1 < \Im(\rho) < T_2}} m_\rho \, \zeta(\rho + i y_1) \zeta(1 - \rho - i y_2) x^\rho
= I_1+I_2+I_3+I_4 - \Delta.
\]

Assuming the Riemann Hypothesis, every non-trivial zero is of the form \(\rho=\tfrac{1}{2}+i\gamma\), and we have
\[
\zeta(1-\rho - i y_2) = \overline{\zeta(\rho + i y_2)}.
\]
This gives the claimed identity.
\end{proof}

Note that in the proof above we included the factor $m_\rho$, but consistent with Banks' paper, from now on, whenever we write sums over non-trivial zeros of $\zeta$, we may omit the factor $m_\rho$ from the notation and these are always understood to be taken with multiplicity.

\medskip 

In the subsequent sections, we calculate each $I_j$.

\subsection{Estimation of \(I_1\)} Using the functional equation \( \zeta(s) = \zeta(1 - s) \mathcal{X}(s) \), we can express \( \zeta(1 - s - i y_2) \) as follows

\[
\zeta(1 - \sigma - it - iy_2) = \zeta(\sigma + i(t + y_2)) \mathcal{X}(1 - \sigma - it - iy_2). \tag{4.5} \label{eq:4.5}
\]

By substituting \eqref{eq:4.5} back into \eqref{eq:4.1} we have

\[
I_1 = \frac{1}{2 \pi} \int_{T_1}^{T_2} \frac{\zeta'(c + it)}{\zeta(c + it)} \zeta(c + i(t + y_1)) \zeta(c + i(t + y_2)) \mathcal{X}(1 - c - it - i y_2) x^{c + it} \, dt
\]

\[
= \frac{1}{2 \pi} \sum_{n = 1}^{\infty} \sum_{a_1 a_2 a_3 = n} \left( -\Lambda(a_1) a_2^{-i y_1} a_3^{-i y_2} \right) \int_{T_1}^{T_2} \left( \frac{x}{n} \right)^{c + it} \mathcal{X}(1 - c - it - i y_2)  \, dt, \tag{4.6} \label{eq:4.6}
\]
where we have written the Dirichlet series for the zeta function and its logarithmic derivative, since \(c > 1\).

By an adaptation of the calculation in Lemma \ref{lem:3.7}, valid for negative values of \(t\) once the shift by \(y_2\) is included, we obtain

\[
\mathcal{X}(1 - c - it - i y_2) =
\]

\[
2^{1 - c - i(t + y_2)} \pi^{-c - i(t + y_2)} \sqrt{2 \pi} (c + i(t + y_2))^{c + i(t + y_2) - \frac{1}{2}} e^{-c - i(t + y_2)} \left\{ 1 + O\left( \frac{1}{t + y_2} \right) \right\}
\]

\[
\times \frac{1}{2i} e^{\frac{\pi i}{2} (1 - c)} e^{\frac{\pi}{2}(t + y_2)} \left\{ 1 + O\left( e^{-\frac{\pi}{2}(t + y_2)} \right) \right\}
\]

\[
= \left( \frac{t + y_2}{2 \pi} \right)^{c - \frac{1}{2}} e^{-i \pi / 4} \exp\left( i(t + y_2) \log\left( \frac{t + y_2}{2 \pi e} \right) \right) \left\{ 1 + O\left( \frac{1}{t + y_2} \right) \right\}. 
\] 

By substituting this expression into \eqref{eq:4.6}, we obtain

\[
I_1 = \frac{1}{2 \pi} \sum_{n=1}^{\infty} \sum_{a_1 a_2 a_3 = n} \left( -\Lambda(a_1) a_2^{-i y_1} a_3^{-i y_2} \right) \times
\]

\[
\int_{T_1 + y_2}^{T_2 + y_2} \left( \frac{x}{n} \right)^{c + i(t - y_2)} \left( \frac{t}{2 \pi} \right)^{c - \frac{1}{2}} e^{-i \pi / 4} \exp\left( i t \log \left( \frac{t}{2 \pi e} \right) \right) \{ 1 + O(t^{-1}) \} \, dt
\]

\[
= \frac{1}{2 \pi} e^{-i \pi / 4} \sum_{n=1}^{\infty} \left( \frac{x}{n} \right)^{c - i y_2} \sum_{a_1 a_2 a_3 = n} \left( -\Lambda(a_1) a_2^{-i y_1} a_3^{-i y_2} \right) \times
\]

\[
\int_{T_1 + y_2}^{T_2 + y_2} \left( \frac{t}{2 \pi} \right)^{c - \frac{1}{2}} \exp \left( i t \log \left( \frac{t x}{2 \pi e n} \right) \right) \{ 1 + O(t^{-1}) \} \, dt. \tag{4.7} \label{eq:4.7}
\] 

By applying Lemma \ref{lem:3.8}, the integral in the right hand side of the last equality in \eqref{eq:4.7}  evaluates to

\[
(2 \pi)^{1 - c} \left( \frac{2 \pi n}{x} \right)^c e^{i \pi / 4 - i \frac{2 \pi n}{x}} \mathbf{1}(T_1 + y_2, T_2 + y_2; \frac{2 \pi n}{x})
\]

\[
+ O\left( (T_1 + y_2)^{c - \frac{1}{2}} + \frac{(T_1 + y_2)^{c + \frac{1}{2}}}{|T_1 + y_2 - \frac{2 \pi n}{x}| + (T_1 + y_2)^{\frac{1}{2}}} + \frac{(T_2 + y_2)^{c + \frac{1}{2}}}{|T_2 + y_2 - \frac{2 \pi n}{x}| + (T_2 + y_2)^{\frac{1}{2}}} \right),
\]
where \( \mathbf{1}(a, b; u) = 1 \) for \( a < u \leq b \) and \( \mathbf{1}(a, b; u) = 0 \) otherwise.

Therefore, we can express \( I_1 \) as \( I_1 = K_1 + O(K_2) \) where

\[
K_1 = \sum_{\frac{(T_1 + y_2)x}{2 \pi} < n \leq \frac{(T_2 + y_2)x}{2 \pi}} e^{-i y_2 \log \frac{x}{n} - i \frac{2 \pi n}{x}} \sum_{a_1 a_2 a_3 = n} \left( -\Lambda(a_1) a_2^{-i y_1} a_3^{-i y_2} \right), \tag{4.8} \label{eq:4.8}
\]
and

\[
K_2 = \sum_{n=1}^{\infty} \left( \frac{x}{n} \right)^c \left| \sum_{a_1 a_2 a_3 = n} \left( -\Lambda(a_1) a_2^{-i y_1} a_3^{-i y_2} \right) \right| \times 
\]

\[
\left\{ (T_1 + y_2)^{c - \frac{1}{2}} + \frac{(T_1 + y_2)^{c + \frac{1}{2}}}{|T_1 + y_2 - \frac{2 \pi n}{x}| + (T_1 + y_2)^{\frac{1}{2}}} + \frac{(T_2 + y_2)^{c + \frac{1}{2}}}{|T_2 + y_2 - \frac{2 \pi n}{x}| + (T_2 + y_2)^{\frac{1}{2}}} \right\}. \tag{4.9} \label{eq:4.9}
\]

\subsubsection{Bound on $K_1$} We begin by stating the main reduction of this section as

\begin{theorem}
\label{thm:K1Reduction}
Let \(T\) be a large parameter and let \(x\) be a prime satisfying
\[
x \asymp \log T .
\]
Let \(y_1, y_2\) be fixed real numbers. Define
\[
K_1 := -x^{-i y_2} \left( K(T_2) - K(T_1) \right),
\]
where
\[
K(T) := \sum_{\substack{a_1, a_2, a_3 \\ a_1 a_2 a_3 \leq \frac{x (T + y_2)}{2 \pi}}}
\Lambda(a_1)\, a_1^{i y_2}\, a_2^{i (y_2 - y_1)}\,
e\!\left( -\frac{a_1 a_2 a_3}{x} \right).
\tag{4.10}\label{eq:4.10}
\]
Then
\[
K_1 = O\!\left( T \log^{1/2} T \, \exp\!\bigl( - C \sqrt{\log T} \bigr) \right),
\]
where the implied constant depends at most on \(y_1\) and \(y_2\).
\end{theorem}

The rest of this subsection is devoted to the detailed proof of Theorem~\ref{thm:K1Reduction}.

Consider the sum in \eqref{eq:4.10}. Our goal is to evaluate this sum, and consequently to evaluate \( K_1 \), as it follows from \eqref{eq:4.8} that \( K_1 = -x^{-i y_2} (K(T_2) - K(T_1)) \). For convenience, we introduce the notation \( X := \frac{x (T_2 + y_2)}{2 \pi} \).

Recall that \( x \) is a prime. We consider two cases: whether \( x \) divides \( a_1a_2a_3 \) or not. When \( x \mid a_1a_2a_3 \), the sum \( K \) is estimated using the approach described later in \S\ref{subsec:4.4.1}. In this case, \( K(T_2) \) and therefore \( K_1 \) is \( \ll  T \exp\!\big(-C\sqrt{\log T}\big) \), where we will use the fact that the proportion of integers up to \( n \) that are divisible by a prime \( x \) is asymptotically \( \sim x^{-1} \).

Now, we consider the case \( x \nmid a_1 a_2 a_3 \). Using the standard identity (see, for example, Section 9.2 of \cite{MontgomeryVaughan2007} and the related exercises) that for \( (n, x) = 1 \), 
\[
e\left( \frac{n}{x} \right) = \frac{1}{\varphi(x)} \sum_{\chi \in \widehat{\mathbb{Z}/x\mathbb{Z}}} \chi(n) \tau \left( \overline{\chi} \right),
\]
where \( \tau(\chi) \) denotes the Gauss sum, we can express \( K(T_2) \) as

\[
K(T_2) = \frac{1}{\varphi(x)} \sum_{\chi \in \widehat{\mathbb{Z}/x\mathbb{Z}}} \tau \left( \overline{\chi} \right) \sum_{\substack{a_1, a_2, a_3 \\ a_1 a_2 a_3 \leq X \\ x \nmid a_1 a_2 a_3}} \Lambda(a_1) a_1^{i y_2} a_2^{i (y_2 - y_1)} \chi(-a_1 a_2 a_3)
\]

\[
= \frac{1}{\varphi(x)} \sum_{\chi \in \widehat{\mathbb{Z}/x\mathbb{Z}}} \chi(-1) \tau \left( \overline{\chi} \right) \sum_{\substack{a_1, a_2, a_3 \\ a_1 a_2 a_3 \leq X \\ x \nmid a_1 a_2 a_3}} \Lambda(a_1) \chi(a_1) a_1^{i y_2} \chi(a_2) a_2^{i (y_2 - y_1)} \chi(a_3). \tag{4.11} \label{eq:4.11}
\] 
The inner sum can be evaluated using Perron's formula.

\begin{lemma}
    
\label{lem:4.1}

(the truncated Perron's formula) Let \[
\alpha(s) = \sum_{n=1}^{\infty} b_n n^{-s}. \tag{4.12} \label{eq:4.12}
\]

If \(\sigma_0 > \max(0, \sigma_a)\), \(\sigma_a\) being the abscissa of absolute convergence of the Dirichlet series \(\alpha(s)\), and \(X > 0\), then 

\[
\sum_{n \leq X} b_n = \frac{1}{2\pi i} \int_{\sigma_0 - iW}^{\sigma_0 + iW} \alpha(s) \, \frac{X^s}{s} \, ds + R_1(X,W), \tag{4.13} \label{eq:4.13}
\]
where 

\[
R_1(X,W) \ll \sum_{X/2 < n < 2X} |b_n| \min \left\{ 1, \frac{X}{W |X - n|} \right\} + \frac{(4X)^{\sigma_0}}{W} \sum_{n=1}^{\infty} \frac{|b_n|}{n^{\sigma_0}} + 1. \tag{4.14} \label{eq:4.14}
\]

\end{lemma}

\begin{proof}

\cite[Theorem 5.2 and Corollary 5.3]{MontgomeryVaughan2007}.
    
\end{proof}

For our case study, \( \sigma_0 = 1 + \frac{1}{\log X} \). We begin by bounding the first sum of the error term \( R_1(X, W) \) in \eqref{eq:4.14}. For the coefficients \( b_n \), by Lemma \ref{lem:3.14}, we have

\[
|b_n| = \left| \sum_{\substack{a_1, a_2, a_3 \\ a_1 a_2 a_3 = n}} \Lambda(a_1) a_1^{-iy_1} a_3^{i(y_2 - y_1)} e^{-i \frac{2 \pi}{x} a_1 a_2 a_3} \right| \ll d(n) \log n. \tag{4.15} \label{eq:4.15}
\]

In the range of summation in the first sum for \( R_1(X, W) \) in \eqref{eq:4.14}, i.e. in the range \( X/2 < n < 2X \), if \( 1 \)  is the minimum, the inequality \( 1 \le \frac{X}{W |X - n|} \) implies \( \frac{X}{2} < X - \frac{X}{W} \le n \le X + \frac{X}{W} < 2X \).

If \( 1 \) is not the minimum, then we have a union of two intervals, \( X + \frac{X}{W} < n < 2X \) and \( \frac{X}{2} < n < X - \frac{X}{W} \), for the range of summation. Combining these intervals, we get 

\[
\sum_{\frac{X}{2} < n < 2X} |b_n| \min \left\{ 1, \frac{X}{W |X - n|} \right\} = \frac{1}{2} \sum_{\frac{X}{2} < n < 2X} d(n) \log n \min \left\{ 1, \frac{X}{W |X - n|} \right\}.
\]

\[
\ll \sum_{X - \frac{X}{W} \le n \le X + \frac{X}{W}} d(n) \log n 
+ \frac{X}{W} \sum_{\frac{X}{2} < n < X - \frac{X}{W}} \frac{d(n) \log n}{X - n}
+ \frac{X}{W} \sum_{X + \frac{X}{W} < n < 2X} \frac{d(n) \log n}{n - X}. \tag{4.16} \label{eq:4.16}
\]

Next, we estimate the sums individually. By the Partial Summation Formula, we have  

\[
\sum_{X - \frac{X}{W} \le n \le X + \frac{X}{W}} d(n) \log n \ll \left( X + \frac{X}{W} \right) \log \left( X + \frac{X}{W} \right) - \left( X - \frac{X}{W} \right) \log \left( X - \frac{X}{W} \right), \tag{4.17} \label{eq:4.17}
\]
which, by Lemma \ref{lem:3.15}, the right-hand side of \eqref{eq:4.17} is \( \ll W^{-1} X \log X \).

Now, we estimate the second sum in \eqref{eq:4.16}. By applying the Partial Summation Formula, we have

\[
\frac{X}{W} \sum_{\frac{X}{2} < n < X - \frac{X}{W}} \frac{d(n) \log n}{X - n} = \left( X - \frac{X}{W} \right) \log^2 \left( X - \frac{X}{W} \right)
\]

\[
+ (2\gamma - 1) \left( X - \frac{X}{W} \right) \log \left( X - \frac{X}{W} \right) + O \left( X^{1/3} \log X \right)
\]

\[
- \frac{X}{W} \log^2 \left( \frac{X}{2} \right) - \frac{X}{W} (2\gamma - 1) \log \left( \frac{X}{2} \right)
+ O \left( \frac{X^{1/3}}{W} \log X \right)
\]

\[
- \frac{X}{W} \int_{\frac{X}{2}}^{X - \frac{X}{W}} \left( \sum_{n \le t} d(n) \right)
\frac{(X - t) \cdot \frac{1}{t} + \log t}{(X - t)^2}. \tag{4.18} \label{eq:4.18}
\]
By substituting
$$
\sum_{n \le t} d(n) = t \log t + (2\gamma - 1)t + O(t^{1/3}),
$$ 
into the final line of \eqref{eq:4.18}, and applying similar techniques along with lengthy and tedious calculations, we obtain the bound $O\left(W^{-1} X \log^2 X\right)$ for \eqref{eq:4.18}. Similarly, the third sum in \eqref{eq:4.16} is also \( O\left(W^{-1} X \log^2 X\right) \). 

Now we aim to bound the series in the second term in the right hand side of \eqref{eq:4.14}. But this series equals

\[
\sum_{n=1}^{\infty} \frac{\log n \cdot d(n)}{2 n^a} = - \zeta'(a) \cdot \zeta(a) \ll \log^3 X, \tag{4.19} \label{eq:4.19}
\]
since \( - \zeta'(a) \ll \log^2 X \) and \( \zeta(a) \ll \log X \) for \( a = 1 + \log^{-1} X \). Therefore, the ultimate contribution from \eqref{eq:4.14} is

\[
R_1(X, W) \ll \frac{X}{W} \log X + \frac{X}{W} \log^2 X + \frac{X}{W} \log^3 X  + 1 \ll \frac{X}{W} \log^3 X. \tag{4.20} \label{eq:4.20}
\]

Now, we turn to estimating the integral contribution in \eqref{eq:4.13}. Consider the rectangular contour with paths defined as follows: let $\ell_1$ denote the horizontal segment from $b - Wi$ to $a - Wi$, $\ell_2$ the vertical segment from $b + Wi$ to $b - Wi$, $\ell_3$ the horizontal segment from $a + Wi$ to $b + Wi$, and $\ell_4$ the vertical segment from $a - Wi$ to $a + Wi$.

By the definition of \( X \), we have \( X \geq e \). We choose \( a = 1 + \frac{1}{\log X} \) and set \( b = 1 - \frac{c}{\log xW} > \frac{1}{2} \), where \( c \) is the constant from Lemma \ref{lem:3.10}. Additionally, \( W \) must be sufficiently large to accommodate the shifts \( y_2 \) and \( y_2 - y_1 \).

The function

\[
k_X(s, \chi) = \frac{L'}{L}(s - iy_2, \chi) L(s - i(y_2 - y_1), \chi) L(s, \chi) \frac{X^s}{s} \tag{4.21} \label{eq:4.21}
\]
has simple poles at $ s = 1 + iy_2$, $ 1 + i(y_2 - y_1)$, and $1$ for the principal character, and at $s = \beta + iy_2$ for a real quadratic character $\chi$, if such a zero $\beta$ of $L(s, \chi)$ exists. Otherwise, it is analytic inside the region of integration for all other characters.

For the principal character, noting that the poles of \( k_X(s, \chi) \) are distinct, the sum of the residues inside the rectangle is given by

\[
\sum \text{Res} \left( \frac{L'(s - iy_2, \chi)}{L(s - iy_2, \chi)} L(s - i(y_2 - y_1), \chi) L(s, \chi) \frac{X^{s}}{s} \right)
\]

\[
= - L(s - i(y_2 - y_1), \chi) L(s, \chi) X^{s} \frac{1}{s} \bigg|_{s = 1 + iy_2}
\]

\[
+ \frac{\varphi(x)}{x} \frac{L'(s - iy_2, \chi)}{L(s - iy_2, \chi)} L(s, \chi) X^{s} \frac{1}{s} \bigg|_{s = 1 + i(y_2 - y_1)}
\]

\[
+ \frac{\varphi(x)}{x} \frac{L'(s - iy_2, \chi)}{L(s - iy_2, \chi)} L(s - i(y_2 - y_1), \chi) X^{s} \frac{1}{s} \bigg|_{s = 1}
\]

\[
= - L(1 + i(y_1), \chi) L(1 + iy_2, \chi) X^{1 + iy_2} \frac{1}{1 + iy_2}
\]

\[
+ \frac{\varphi(x)}{x} \frac{L'(1 + i(y_1 - y_2), \chi)}{L(1 + i(y_1 - y_2), \chi)} L(1 + i(y_2 - y_1), \chi) X^{1 + i(y_2 - y_1)} \frac{1}{1 + i(y_2 - y_1)}
\]

\[
+ \frac{\varphi(x)}{x} \frac{L'(1 - iy_2, \chi)}{L(1 - iy_2, \chi)} L(1 - i(y_2 - y_1), \chi) X. \tag{4.22} \label{eq:4.22}
\]

We denote this sum as \( S_1(X) \). From the Cauchy Integration Formula, we have the relation

\[
\int_{\ell_4} k_X(s, \chi) \, ds - S_1(X) = - \int_{\ell_1}k_X(s, \chi) \, ds - \int_{\ell_2} k_X(s, \chi) \, ds - \int_{\ell_3} k_X(s, \chi) \, ds := R_2(X, W), \tag{4.23} \label{eq:4.23}
\]
where we need to take \( W \) large enough to encompass the shifts of the L-function and the logarithmic derivative of the L-function within the interior of the contour. This is possible because \( W \) will later be chosen as an increasing function of \( T \) and thus can be made arbitrarily large by taking \( T \) sufficiently large.

Estimating the three integrals for \( R_2(X, W) \) in \eqref{eq:4.23} follows the same procedure as for the case of \( L_1 \) in \S\ \ref{subsec:4.4.1}. This is because for \( \Re(s) > 1 \), we have

\[
L(s, \chi_0) = \zeta(s) \left( 1 - \frac{1}{x^s} \right),
\]
and both \( \zeta(s) \) and \( 1 - \frac{1}{x^s} \) have analytic continuation beyond the \( 1 \)-line (except for the common pole for \( \zeta(s) \) and \( L(s, \chi_0) \)); therefore, their product also has analytic continuation. For large $x$, we have \( 1 - \frac{1}{x^s} \asymp 1 \) and thus 

\[
\left| L(s, \chi_0) \right| \asymp \left| \zeta(s) \right|.
\]
Similarly, by straightforward calculation, we obtain

\[
\frac{L'(s, \chi_0)}{L(s, \chi_0)} = \frac{\zeta'(s)}{\zeta(s)} + \frac{\log(x) x^{-s}}{1 - \frac{1}{x^s}},
\]
which implies that

\[
\left| \frac{L'(s, \chi_0)}{L(s, \chi_0)} \right| \asymp \left| \frac{\zeta'(s)}{\zeta(s)} \right|.
\]

Thus, we have the following estimates (see Lemmas \ref{lem:4.6}, \ref{lem:4.7}, and \ref{lem:4.8}).

\[
\int_{b - Wi}^{a - Wi} \frac{L'}{L}(s + iy_1, \chi_0) L(s - i(y_2 - y_1), \chi_0) L(s,\chi_0) \frac{X^{s}}{s} \, ds \ll \frac{X}{W} \log^6 W,
\]

\[
\int_{b + Wi}^{b - Wi} \frac{L'}{L}(s + iy_1, \chi_0) L(s - i(y_2 - y_1), \chi_0) L(s,\chi_0) \frac{X^{s}}{s} \, ds \ll X^b \log^5 W,
\]

\[
\int_{a + Wi}^{b + Wi} \frac{L'}{L}(s + iy_1, \chi_0) L(s - i(y_2 - y_1), \chi_0) L(s,\chi_0) \frac{X^{s}}{s} \, ds \ll \frac{X}{W} \log^6 W.
\]

Therefore, for the entire error term, we have
\[
R_1(X, W) + R_2(X, W) \ll \frac{X}{W} \log^3 X + \frac{X}{W} \log^6 W + X^b \log^5 W, \tag{4.24} \label{eq:4.24}
\]
and if we choose 
\[
W = \exp\big(C' \sqrt{\log T}\,\big),
\] 
for some \(C'>2C\), where \(C\) is the constant in Theorem~\ref{thm:1.1}, then this is analogous to \eqref{eq:4.54}. Note that the exponent \(b\) in the last term is now
\[
b = 1 - \frac{c}{\log(x W)}
  = 1 - \frac{c}{\log W} \left( 1 + O\Big(\frac{\log x}{\log W}\Big) \right)
  = 1 - \frac{c}{C' \sqrt{\log T}} \left( 1 + o(1) \right),
\]
so it differs from the \(L_1\)-case (see \S\ref{subsec:4.4.1}) only by a subpolynomial factor. Therefore, the same calculations as in the \(L_1\)-case show that the combined contribution of all three terms is
\[
R_1(X, W) + R_2(X, W) \ll T \exp\!\big(-C \sqrt{\log T}\,\big),
\]
which is asymptotically negligible compared to the main term in \eqref{eq:2.3}, as desired.

In addition, for the principal character we have \(S_{1}(X_{i}) \ll X_{i}\). Using the same reasoning as in the \(L_{1}\)-case in \S\ref{subsec:4.4.1}, we write
\[
S_{1}(X_{2}) - S_{1}(X_{1})
  = (X_{2} - X_{1})\, S_{1}'(\xi)
  \ll |X_{2} - X_{1}|
  \ll T \log T \, e^{-C\sqrt{\log T}}, \tag{4.25} \label{eq:4.25}
\]
where \(\xi\) lies between \(X_{1}\) and \(X_{2}\), and
\[
X_{j} = \frac{x(T_{j} + y_{2})}{2\pi}, 
\qquad
T_{1} = T, \qquad T_{2} = (1+\epsilon)T,
\qquad 
\epsilon = \exp(-C\sqrt{\log T}).
\]

Combining these contributions, we deduce
\[
S_{1}(X_{2}) - S_{1}(X_{1})
+ R_{1}(X_{2},W) - R_{1}(X_{1},W)
+ R_{2}(X_{2},W) - R_{2}(X_{1},W)
\ll T \log T \, e^{-C\sqrt{\log T}}.
\]

Multiplying by the factor \(\varphi(x)^{-1} \asymp x^{-1} \asymp (\log T)^{-1}\) coming from \eqref{eq:4.11}, we obtain
\[
\frac{1}{\varphi(x)} \Big(
S_{1}(X_{2}) - S_{1}(X_{1})
+ R_{1}(X_{2},W) - R_{1}(X_{1},W)
+ R_{2}(X_{2},W) - R_{2}(X_{1},W)
\Big)
\ll T e^{-C\sqrt{\log T}}.
\]

Since the factor \(\varphi(x)^{-1}\) provides decay \(\asymp (\log T)^{-1}\), the contribution from \(K(\chi_{0})\), and therefore from \(K_{1}(\chi_{0})\), is negligible compared to the main term of Theorem~\ref{thm:1.1}; that is,
\[
K_{1}(\chi_{0}) \ll T e^{-C\sqrt{\log T}}
  = o\!\left( T \log T \, e^{-C\sqrt{\log T}} \right),
\]
where we used \(\varphi(x) = x-1 \ll x\) and the well-known fact that \(\tau(\chi_{0}) = \mu(x) = -1\).

We now estimate \(K(\chi)\) for non-principal characters \(\chi\). First consider the case where \(\chi\) is real; the estimation for non-real non-principal characters is identical, except that a possible real exceptional zero need not be considered. The bound for \(R_{1}(X,W)+R_{2}(X,W)\) is the same as in the principal-character case. For the remainder of the argument we use the same rectangular contour as before, allowing a slight adjustment of the left abscissa \(b\) so that the contour does not pass too close to any exceptional zero.

\medskip
\textit{Case 1.} No Exceptional Zero.  
In this case, we take \( b = 1 - \frac{c}{\log x W} \), where \(c\) is the constant from Lemma~4. The second equality follows from the change in notation. Since there are no poles in the function
\[
k_{X}(s,\chi)
  = \frac{L'}{L}(s - iy_{2}, \chi)
    L(s - i(y_{2} - y_{1}), \chi)
    L(s, \chi)
    \frac{X^{s}}{s},
\]
the residue is zero.

\medskip
\textit{Case 2.} Let \( \beta \) be an exceptional zero. Depending on the location of this zero, we define  
\[
b = 1 - \frac{c}{A \log x W},
\]
and choose \(A\) such that \(|b - \beta| > \log x\). If the zero lies inside the rectangle, then the residue is
\[
\operatorname{Res}\!\left(
\frac{L'(s - iy_{2}, \chi)}{L(s - iy_{2}, \chi)}
L(s - i(y_{2} - y_{1}), \chi)
L(s, \chi)
\frac{X^{s}}{s}
\right)
\]
\[
= L(s - i(y_{2} - y_{1}), \chi)
  L(s, \chi)
  X^{s}
  \frac{1}{s} \bigg|_{s = \beta + iy_{2}}
\]
\[
= L(\beta + iy_{1}, \chi)
  L(\beta + iy_{2}, \chi)
  X^{\beta + iy_{2}}
  \frac{1}{\beta + iy_{2}}. \tag{4.26}\label{eq:4.26}
\]

We denote this by \(S_{2}(X)\). Using the bound
\[
\beta \le 1 - \frac{c}{x^{1/2} (\log x)^{2}}
\quad\text{(see Lemma~\ref{lem:3.11})},
\]
we obtain
\[
S_{2}(X)
  \ll (xT)^{1 - \frac{c}{x^{1/2} (\log x)^{2}}},
\]
which is \(T \log T\) when \(x \asymp \log T\), by a straightforward calculation. Since \(\beta < 1\), we have \(S_{2}(X) \asymp S_{1}(X)\), and thus we may use the same estimate as in \eqref{eq:4.25}:
\[
S_{2}(X_{2}) - S_{2}(X_{1})
  = (X_{2} - X_{1}) S_{2}'(\xi)
  \ll |X_{2} - X_{1}|.
\]

Multiplying this by the two coefficients in \eqref{eq:4.11}, namely \(\varphi(x)^{-1} \ll x^{-1}\) and \(|\tau(\overline{\chi})| \le x^{1/2}\), we obtain
\[
S_{2}(X)
  \ll T \log^{1/2} T \exp \bigl( - C \sqrt{\log T} \bigr). \tag{4.27} \label{eq:4.27}
\]

Using the Cauchy Integral Formula, we obtain the relation

\[
\int_{\ell_4} k_X(s, \chi) \, ds - S_2(X) = - \int_{\ell_1} k_X(s, \chi) \, ds - \int_{\ell_2} k_X(s, \chi) \, ds - \int_{\ell_3} k_X(s, \chi) \, ds := R_2(X, W), \tag{4.28} \label{eq:4.28}
\]
where \( W \) is chosen large enough to contain the shifts of the L-function and its logarithmic derivative inside the contour. Note that \( S_2(X) = 0 \) in Cases 1, and it is \( \ll T \log^{1/2} \exp \bigl( - C \sqrt{\log T} \bigr) \) in Case 2.

Now we estimate the integrals in the right hand side of \eqref{eq:4.28}.

\begin{lemma}
\label{lem:4.2}    

For all non-principal characters and for \( b = 1 - \frac{c}{A \log x W} \), where \( A \) is defined such that \( |b - \beta| > \frac{1}{\log x} \), we have the following estimate

\[
\int_{\ell_1} \frac{L'}{L}(s - iy_2, \chi) L(s - i(y_2 - y_1), \chi) L(s, \chi) \frac{X^{s}}{s} \, ds \ll \frac{X}{W} \cdot \log^3 xW \cdot \log X. 
\]

\end{lemma}

\textbf{Remark.} Note that knowing the location of the hypothetical zero is not crucial, as the estimates in this lemma and the subsequent ones hold for any absolute constant \( A \).

\begin{proof}
    
We make use of Lemmas~\ref{lem:3.12} and~\ref{lem:3.13}.
Along the line \( \ell_1 \), the minimum of \( \frac{1}{|\sigma - 1|} \) occurs at the endpoints. Therefore, it is either \( \frac{1}{|a - 1|} = \log X \) or \( \frac{1}{|b - 1|} = A \log x W \). Therefore, the minimum is \( O(\log x W) \) for a similar choice of \( x \) and \( W \) as before. In addition, on \( \ell_1 \), the expression \( |xt|^{1 - \sigma} + 1 \) is maximized when \( \sigma = b \). Therefore, on \( \ell_1 \), we have the bound

\[
|xt|^{1 - \sigma} + 1 \ll (xW)^{\frac{1}{A \log xW}} + 1 = O(1).
\]

Using these results, we conclude that both \( L(s) \) and \( L(s - i(y_2 - y_1)) \) are \( \ll \log x W \) on \( \ell_1 \). For \( \frac{L'}{L}(s - i y_2) \), we similarly have the estimate \( O( \log x W ) \). Finally, we have \( \frac{X^s}{s} \ll \frac{X}{W} \). Thus, the desired result follows.

\end{proof} 

\begin{lemma}
\label{lem:4.3} 
    
If \( \chi \) is a non-principal, non-quadratic character, or if \( \chi \) is quadratic with no exceptional zeros, then we have the estimate:

\[
\int_{\ell_2} \frac{L'}{L}(s - iy_2, \chi) L(s - i(y_2 - y_1), \chi) L(s, \chi) \frac{X^{s}}{s} \, ds \ll X^b \log^4 xW.
\]

\end{lemma}

\begin{proof}
    
On \( \ell_2 \), using Lemma~\ref{lem:3.12}, we obtain the bound \( L(s) \ll (x |t|)^{1/(A \log x W)} \log x W \). Moreover, using Lemma~\ref{lem:3.13} we find that \( \frac{L'}{L}(s) \ll \log x |t| \leq \log x W \). Additionally, for all \( t \) on \( L_2 \), we have \( \frac{X^s}{s} \ll X^b \min(1, t^{-1}) \). A straightforward calculation then completes the proof of the lemma.

\end{proof}

The calculation for line \( \ell_3 \) proceeds similarly to that for \( \ell_1 \). Thus, for the final integral, we have

\begin{lemma}
\label{lem:4.3}

For all non-principal characters, whether \( b = 1 - \frac{c}{5 \log x W} \) or \( b = 1 - \frac{c}{3 \log x W} \), we have the estimate

\[
\int_{\ell_3} \frac{L'}{L}(s - iy_2, \chi) L(s - i(y_2 - y_1), \chi) L(s, \chi) \frac{X^{s}}{s} \, ds \ll \frac{X}{W} \cdot \log^3 xW \cdot \log X. 
\]

\end{lemma}

Thefore, by similar calculation as we did for the principal character, we finally conclude that \(
K_1(\chi) \ll T \log^{1/2} \exp \bigl( - C \sqrt{\log T} \bigr), 
\) for the quadratic character, where we used the fact that \( \varphi(x) = x - 1 \asymp x \) and also the well-known fact that for a Dirichlet non-principal character $\chi$ modulo 
$x$, \( |\tau(\chi)| \le x^{1/2} \).

Finally, we turn to the estimation of \( K(\chi) \) for non-principal complex characters. The approach for this case follows the same reasoning as for real non-principal characters with no exceptional zeros. However, caution must be exercised, as there are \( \phi(x) - 2 = x - 3 \ll x \) such characters, introducing an additional \( \ll x \) factor. Therefore, we will have a factor of \( \varphi(x)^{-1} \cdot |\tau(\chi)| \cdot (\varphi(x) - 2) \ll x^{1/2} \).  

Since the inner triple sum in \eqref{eq:4.11} is given by \(
R_1(X, W) + R_2(X, W),
\) where \( R_1 \) is given by \eqref{eq:4.20} and \( R_2 \) by \eqref{eq:4.28}, then by a similar choice of \( W \), we have

\[
\sum_{\chi = \text{Complex Characters}} K_1(\chi) \ll x^{1/2} \left( \frac{X}{W} \log^3 X + \frac{X}{W} \cdot \log^3 (xW) \cdot \log X + X^b \log^4 (xW) \right), \tag{4.29} \label{eq:4.29}
\]
which is \( \ll T \log^{1/2} T \exp \bigl( - C \sqrt{\log T} \bigr)\).

Putting everything together, we conclude that
\[
K_1 \ll T \exp\bigl(-C \sqrt{\log T}\bigr) + T \log^{1/2} T \, \exp\bigl(-C \sqrt{\log T}\bigr)
\ll T \log^{1/2} T \, \exp\bigl(-C \sqrt{\log T}\bigr),
\]
which completes the proof of Theorem~\ref{thm:K1Reduction}.

\subsubsection{Bound on $K_2$} 

Now we estimate \( K_2 \) from \eqref{eq:4.9}. Applying Lemma~\ref{lem:3.14}, we obtain

\[
K_2 \ll x T^{c - \frac{1}{2}} \sum_{n=1}^{\infty} \frac{d(n) \log n}{n^c} + x^2 T^{c + \frac{1}{2}} \sum_{n \geqslant 2} \frac{d(n) \log n}{n^c \left( |xt - 2\pi n| + xt^{1/2} \right)}. \tag{4.30} \label{eq:4.30}
\]

Since \( c = 1 + \frac{1}{\log x} \), the first sum simplifies as

\[
\sum_{n=1}^{\infty} \frac{d(n) \log n}{n^c} = -2 \zeta'(c)\zeta(c) \ll \log^3 x. \tag{4.31} \label{eq:4.31}
\]

For the second sum, we apply Lemma~\ref{lem:3.19}. Putting these together, we ultimately obtain the bound

\[
K_2 \ll x T^{c - \frac{1}{2}} \log^3 x + x T^{1/2+c} \left( \frac{\log^3 T + \log^3 x}{T} \right) \ll x T^{c - \frac{1}{2}} \left(\log^3 T + \log^3 x \right). 
\]

With \(x\asymp\log T\) and \(c=1+\dfrac{1}{\log x}\) this becomes
\[
K_2 \ll T^{\tfrac12+\frac{1}{\log\log T}}\,\log^4 T. \tag{4.32}
\label{eq:4.32}
\]
Comparing with \(f(T):=T e^{-C\sqrt{\log T}}\) yields
\[
\frac{K_2}{f(T)} \ll T^{-1/2+\frac{1}{\log\log T}}\log^4 T\,e^{C\sqrt{\log T}}\to 0,
\]
so \(K_2=o \big(T e^{-C\sqrt{\log T}}\big)\). 

\subsection{Estimating \( I_2 \) and \( I_4 \)}
For \( I_2 \) and \( I_4 \), by applying Lemma~\ref{lem:3.2}, we find that

\[
I_2 = \frac{1}{2 \pi i} \int_{c + iT_2}^{b + iT_2} \frac{\zeta'(s)}{\zeta(s)} \zeta(s + i y_1) \zeta(1 - s - i y_2) x^s \, ds 
\]

\[
= \frac{1}{2 \pi i} \sum_{\rho = \frac{1}{2} + i \gamma, \, |\gamma - T_2| < 1} \int_{c + iT_2}^{b + iT_2} \frac{\zeta(s + i y_1) \zeta(1 - s - i y_2)}{s - \rho} x^s \, ds + O\left(M^2 \log T \int_b^c x^{\sigma} \, d\sigma \right), \tag{4.33} \label{eq:4.33}
\]
where \( M := \max \left\{ |\zeta(\sigma + i T_2)| : b \le \sigma \le c \, \cup \, 1 - c \le \sigma \le 1 - b \right\} \). Since \( M \ll L^{\lambda} \) by Lemma \ref{lem:3.3}, and \( \log T \ll L \), we can express the integral in the error term above as

\[
\int_b^c x^{\sigma} \, d\sigma \asymp x^c \ll x.
\]

Therefore, the error term can be bounded as \( \ll x L^{2\lambda + 1} \).

By applying Cauchy’s theorem, the integral

\[
J_{\rho} := \int_{c + iT_2}^{b + iT_2} \frac{\zeta(s + i y_1) \zeta(1 - s - i y_2)}{s - \rho} x^s \, ds
\]
is equal to

\[
- 2 \pi i \zeta(\rho + i y_1) \zeta(1 - \rho - i y_2) x^{\rho}
\]

\[
+ \left( \int_{b + i(T_2 + 2)}^{b + iT_2} + \int_{c + i(T_2 + 2)}^{b + i(T_2 + 2)} + \int_{c + iT_2}^{c + i(T_2 + 2)} \right) \frac{\zeta(s + i y_1) \zeta(1 - s - i y_2)}{s - \rho} x^s \, ds. \tag{4.34} \label{eq:4.34}
\]

The residue contributes at most \( \ll x^{1/2} L^{2\lambda} \).  
Along the contour we have the bound
\[
\frac{1}{|s-\rho|}
  \ll 
  \begin{cases}
    \log\log T, & \sigma=b,\\[4pt]
    1,          & \text{elsewhere},
  \end{cases}
\]
so the contribution of the integral in \eqref{eq:4.34} is
\[
\ll x L^{2\lambda} \log\log T.
\]
Recalling that
\[
L=\exp\!\Big(\frac{\log T}{\log\log T}\Big),
\qquad\text{hence}\qquad
L^{2\lambda+1}
  = \exp\!\Big((2\lambda+1)\frac{\log T}{\log\log T}\Big)
  = T^{\frac{2\lambda+1}{\log\log T}}.
\]

Using \(x\asymp\log T\) we obtain
\[
x L^{2\lambda+1} \log\log T
\asymp (\log T)\,T^{\frac{2\lambda+1}{\log\log T}}\,\log\log T.
\]

To compare this with \(T e^{-C\sqrt{\log T}}\) it is convenient to take logarithms.
\begin{align*}
\log\big(x L^{2\lambda+1}\log\log T\big)
&= \log x + \frac{(2\lambda+1)\log T}{\log\log T} + \log\log\log T + O(1) \\
&= \frac{(2\lambda+1)\log T}{\log\log T} + O(\log\log T),
\end{align*}
since \(\log x\asymp\log\log T\). On the other hand
\[
\log\big(T e^{-C\sqrt{\log T}}\big)=\log T - C\sqrt{\log T}.
\]
Now \(\dfrac{(2\lambda+1)\log T}{\log\log T}=o(\log T)\) as \(T\to\infty\), therefore
\[
\log\big(x L^{2\lambda+1}\log\log T\big)=o(\log T),
\qquad
\log\big(T e^{-C\sqrt{\log T}}\big)=(1+o(1))\log T,
\]
and the difference tends to \(-\infty\). Hence
\[
x L^{2\lambda+1} \log\log T = o \big( T e^{-C\sqrt{\log T}} \big).
\]
Finally we conclude
\[
I_2 \ll x L^{2\lambda+1} \log\log T
     = o \big( T e^{-C\sqrt{\log T}} \big),
\tag{4.35}\label{eq:4.35}
\]
which can be absorbed in the error term in Theorem~\ref{thm:1.1}, as desired.

By a similar argument, we conclude

\[
I_4  \ll x L^{2\lambda+1} \log \log T = o \big( T e^{-C\sqrt{\log T}} \big). \tag{4.36} \label{eq:4.36}
\]

\subsection{Estimating \( I_3 \)}
Now, for \( I_3 \), by applying the functional equation for the zeta function \eqref{eq:3.2} and performing the change of variables \( s \mapsto 1 - s \), we obtain

\[
I_3 = \frac{1}{2 \pi i} \int_{b + iT_2}^{b + iT_1} \frac{\zeta'(s)}{\zeta(s)} \zeta(s + i y_1) \zeta(1 - s - i y_2) x^s \, ds
\]

\[
= - \frac{1}{2 \pi i} \int_{b' - iT_2}^{b' - iT_1} \left\{ - \frac{\zeta'(s)}{\zeta(s)} + \log \pi - \frac{1}{2} \psi\left( \frac{s}{2} \right) - \frac{1}{2} \psi\left( \frac{1 - s}{2} \right) \right\}
\]

\[
\times \zeta(s - i y_1) \zeta(s - i y_2) \chi(1 - s + iy_1) x^{1 - s} \, ds, \tag{4.37} \label{eq:4.37}
\]
where \( b' = \frac{1}{2} + \frac{1}{\log \log T} \).

Let \( \Omega \) be the union of the horizontal line segments from \( b' - iT_j \) to \( c - iT_j \) for \( j = 1, 2 \). 
For \( s \in \Omega \), the following bounds hold (cf. Banks \cite[§3.2]{Banks2024})

\begin{enumerate}
    \item \( \zeta(s - iy_j) = O(L^{\lambda}) \),
    \item \( \frac{\zeta'(s)}{\zeta(s)} = O(\log T \log \log T) \),
    \item \( \chi(1 - s + iy_1) = O(T^{\sigma - \frac{1}{2}}) \),
    \item \( \psi\!\left( \tfrac{s}{2} \right) = O(\log T) \),
    \item \( \psi\!\left( \tfrac{1 - s}{2} \right) = O(\log T) \),
    \item \( x^{1 - s} = O(T^{1 - \sigma}) \).
\end{enumerate}

Putting all these components together, we conclude that

\[
I_3 = J_1 + J_2 + J_3 + J_4 + O\left( T^{\frac{1}{2}} L^{2\lambda + 1} \right),
\]
where the terms \( J_1 \), \( J_2 \), \( J_3 \), and \( J_4 \) are defined as follows

\[
J_1 := \frac{x}{2 \pi i} \int_{c-iT_2}^{c-iT_1} \frac{\zeta'(s)}{\zeta(s)} \zeta(s - iy_1) \zeta(s - iy_2) \chi(1 - s + iy_1) x^{-s} \, ds, \tag{4.38} \label{eq:4.38}
\]

\[
J_2 := \frac{-x \log \pi}{2 \pi i} \int_{c-iT_2}^{c-iT_1} \zeta(s - iy_1) \zeta(s - iy_2) \chi(1 - s + iy_1) x^{-s} \, ds, \tag{4.39} \label{eq:4.39}
\]

\[
J_3 := \frac{x}{4 \pi i} \int_{c-iT_2}^{c-iT_1} \psi\left( \frac{s}{2} \right) \zeta(s - iy_1) \zeta(s - iy_2) \chi(1 - s + iy_1) x^{-s} \, ds, \tag{4.40} \label{eq:4.40}
\]

\[
J_4 := \frac{x}{4 \pi i} \int_{c-iT_2}^{c-iT_1} \psi\left( \frac{1-s}{2} \right) \zeta(s - iy_1) \zeta(s - iy_2) \chi(1 - s + iy_1) x^{-s} \, ds. \tag{4.41} \label{eq:4.41}
\]

In the following, we evaluate the terms \( J_i \).

\subsubsection{Bound on \( J_1 \)}
\label{subsec:4.4.1}

We begin with the term \( J_1 \).

\[
J_1 = \frac{x}{2 \pi i} \sum_{n=1}^{\infty} \sum_{a_1a_2a_3=n} \left( -\Lambda(a_1) a_2^{iy_1} a_3^{iy_2} \right) \int_{c-iT_2}^{c-iT_1} \chi(1-s+iy_1) (nx)^{-s} \, ds := L_1 + O(L_2), \tag{4.42} \label{eq:4.42}
\]
where

\[
L_1 = x^{1-iy_1} \sum_{\frac{T_1 + y_1}{2 \pi x} < n \leq \frac{T_2 + y_1}{2 \pi x}} n^{-iy_1} \sum_{a_1a_2a_3=n} \left( -\Lambda(a_1) a_2^{iy_1} a_3^{iy_2} \right), \tag{4.43} \label{eq:4.43}
\]
and

\[
L_2 = \sum_{n=1}^{\infty} \frac{d(n) \log n}{n^c} \ \times 
\]
\[
\left\{ (T_1 + y_1)^{c-\frac{1}{2}} + \frac{(T_1 + y_1)^{c+\frac{1}{2}}}{|T_1 + y_1 - 2 \pi n x| + (T_1 + y_1)^{1/2}} + \frac{(T_2 + y_1)^{c+\frac{1}{2}}}{|T_2 + y_1 - 2 \pi n x| + (T_2 + y_1)^{1/2}} \right\}, \tag{4.44} \label{eq:4.44}
\]
where we applied Lemma \ref{lem:3.14}, for \( L_2 \). Expression \eqref{eq:4.44} further simplifies to

\[
L_2 \ll T^{c-\frac{1}{2}} \sum_{n=1}^{\infty} \frac{d(n) \log n}{n^c} + T^{c+\frac{1}{2}} \sum_{n \geq 2} \frac{d(n) \log n}{n^c \left( |t - 2 \pi n x| + t^{1/2} \right)}, \tag{4.45} \label{eq:4.45}
\]
which is similar to the calculation leading to \eqref{eq:4.32}, and thus, as in the case of \(K_{2}\), we obtain
\[
L_{2} = o \left( T e^{-C\sqrt{\log T}} \right). \tag{4.46} \label{eq:4.46}
\]

In the following we prove 

\begin{theorem}
\label{lem:L1-bound}
Let \( x \asymp \log T \), and let \( T_1 < T_2 \) be positive real numbers. 
Set \( T = \tfrac{T_1 + T_2}{2} \). 
Let \( y_1, y_2 \in \mathbb{R} \) be fixed. Define
\[
L_1 = x^{1 - iy_1} \sum_{\frac{T_1 + y_1}{2\pi x} < n \leq \frac{T_2 + y_1}{2\pi x}} 
n^{-iy_1} \sum_{a_1 a_2 a_3 = n} \left( -\Lambda(a_1) a_2^{iy_1} a_3^{iy_2} \right).
\]
Then we have
\[
L_1 \ll T \exp\!\big(-C\sqrt{\log T}\big).
\]
\end{theorem}

The technique to estimate \( L_1 \) is similar to that of \( K_1 \) (see Theorem \ref{thm:K1Reduction}); however, we provide the details. We first estimate the sum 

\[
L'_1 = L'_1(T_2) := -\sum_{\substack{a_1, a_2, a_3 \\ a_1 a_2 a_3 \le \frac{T_2 + y_1}{2 \pi x}}} \Lambda(a_1) a_2^{-iy_1} a_3^{i(y_2 - y_1)}. \tag{4.47} \label{eq:4.47}
\]

Multiplying \( L'_1(T_2) - L'_1(T_1) \) by \( x^{1 - iy_1} \) gives us \( L_1 \) in \eqref{eq:4.43}.

Considering Perron's formula, whether \( \frac{T_2 + y_1}{2 \pi x} \) is an integer or not contributes an \( O(1) \) difference in the final result, which can be absorbed into the final error term. Therefore, we may assume that \( \frac{T_2 + y_1}{2 \pi x} \) is not an integer. For simplicity, we will now rename \( X := \frac{T_2 + y_1}{2 \pi x} \).

Using the truncated version of Perron's formula (Lemma \ref{lem:4.1}), we can express \( L'_1 \) as
\[
L'_1 = \frac{1}{2 \pi i} \int_{\sigma_0 - i W}^{\sigma_0 + i W} 
\frac{\zeta'}{\zeta}(s + iy_1) \zeta(s - i(y_2 - y_1)) \zeta(s) X^{s} \frac{ds}{s} + R_1(X, W),
\]
where \( \sigma_0 = 1 + \frac{1}{\log X} \).  

Similar to the estimation of \( K_1 \), we first bound \( R_1(X, W) \). We have
\[
|b_n| = \left| \sum_{\substack{a_1, a_2, a_3 \\ a_1 a_2 a_3 = n}} 
\Lambda(a_1) a_2^{-iy_1} a_3^{i(y_2 - y_1)} \right| \ll d(n) \log n,
\]
by Lemma \ref{lem:3.14}.

We have the same bound for \( b_n \) here as we had when estimating \( K_1 \), so the entire calculation from \eqref{eq:4.14} to \eqref{eq:4.20} remains valid in the current context, and thus we have \( R_1(X, W) \ll W^{-1} X \log^3 X \). However, observe that in this case, \( X \asymp T/x \), not \( \asymp xT \) as it was in \eqref{eq:4.20}. 

The classical zero-free region for the zeta function states that there is an absolute constant \(c > 0\) such that \(\zeta(s) \neq 0\) for \( \sigma \geq 1 - \frac{c}{\log \tau} \). With this \(c\), we define \( b = 1 - \frac{c}{\log W} \) and set \( a = 1 + \frac{1}{\log X}\), and consider the rectangular contour with paths defined as follows.

\begin{itemize}
  \item[\(\ell_1\):] \( [b - Wi, a - Wi], \)
  \item[\(\ell_2\):] \( [b + Wi, b - Wi], \)
  \item[\(\ell_3\):] \( [a + Wi, b + Wi], \)
  \item[\(\ell_4\):] \( [a - Wi, a + Wi], \)
\end{itemize}
where \(W\) must be sufficiently large to accommodate the shifts \(y_1\) and \(y_2 - y_1\).

The function 

\[
g_X(s) = \frac{\zeta'}{\zeta}(s + iy_1) \zeta(s - i(y_2 - y_1)) \zeta(s) \frac{X^{s}}{s} \tag{4.48} \label{eq:4.48}
\]
has simple poles at \( s = 1 - iy_1 \), \( s = 1 + i(y_2 - y_1) \), and \( s = 1 \).

Noting that the poles of \( g_X(s) \) are distinct, the sum of the residues is given by

\[
\sum \text{Res} \left( \frac{\zeta'}{\zeta}(s + iy_1) \zeta(s - i(y_2 - y_1)) \zeta(s) X^{s} \frac{1}{s} \right)
\]
    
\[
= - \zeta(s - i(y_2 - y_1)) \zeta(s) X^{s} \frac{1}{s} \bigg|_{s = 1 - iy_1}
\]

\[
+ \frac{\zeta'}{\zeta}(s + iy_1) \zeta(s) X^{s} \frac{1}{s} \bigg|_{s = 1 + i(y_2 - y_1)}
\]

\[
+ \frac{\zeta'}{\zeta}(s + iy_1) \zeta(s - i(y_2 - y_1)) X^{s} \frac{1}{s} \bigg|_{s = 1}
\]

\[
= - \zeta(1 - iy_2) \zeta(1 - iy_1) X^{1 - iy_1} \frac{1}{1 - iy_1} 
\]

\[
+ \frac{\zeta'}{\zeta}(1 + iy_2) \zeta(1 + i(y_2 - y_1)) X^{1 + i(y_2 - y_1)} \frac{1}{1 + i(y_2 - y_1)}
\]

\[
+ \frac{\zeta'}{\zeta}(1 + iy_1) \zeta(1 - i(y_2 - y_1)) X. \tag{4.49} \label{eq:4.49}
\]

For brevity, we denote this sum as \( S_3(X) \).

From the Cauchy Integration Formula, we have the relation

\[
\int_{\ell_4} g_X(s) \, ds - S_3(X) = - \int_{\ell_1} g_X(s) \, ds - \int_{\ell_2} g_X(s) \, ds - \int_{\ell_3} g_X(s) \, ds := R_2(X, W), \tag{4.50} \label{eq:4.50}
\]
where we have taken \( W \leq X \) large enough to encompass the shifts of the zeta function and the logarithmic derivative of the zeta function within the interior of the contour. This is feasible since we can consider \( X = \frac{T_{1,2} + y_1}{2 \pi x} \) to be sufficiently large

Using Lemmas \ref{lem:3.4}, \ref{lem:3.5}, and \ref{lem:3.6}, we now estimate the integrals as follows.

\begin{lemma}
\label{lem:4.6}

We have the estimate

\[
\int_{\ell_1} \frac{\zeta'}{\zeta}(s + iy_1) \zeta(s - i(y_2 - y_1)) \zeta(s) \frac{X^{s}}{s} \, ds \ll \frac{X}{W} \log^6 W. \tag{4.51} \label{eq:4.51}
\]

\end{lemma}

\begin{proof}
    
On the line \( \ell_1 \), the minimum of \( \frac{1}{|\sigma - 1|} \) occurs at the endpoints of the line, and thus is either \( \frac{1}{|a - 1|} = \log X \) or \( \frac{1}{|b - 1|} = c \log W \). Therefore, the minimum is \( O(\log W) \) if we choose $\log X \asymp \log W$. So we may choose the estimate

\[
\min \left\{ \frac{1}{|\sigma - 1|}, \log \log W \right\} \ll \log W.
\]

In addition, on \( \ell_1 \), the expression \( (\log |t|)^{2 - 2\sigma} + 1 \) is maximized when \( \sigma = b \). Therefore, on \( \ell_1 \), we have the bound

\[
(\log |t|)^{2 - 2\sigma} + 1 \ll (\log W)^{\frac{2c}{\log W}} \ll \log W.
\]

As a result, \( \frac{\zeta'}{\zeta}(s) \ll \log^2 W \) on \( \ell_1 \). Similarly, \( \zeta(s) \) and \( \zeta(s - i(y_2 - y_1)) \) are both \( \ll \log^{3/2} W \) (by Lemma \ref{lem:3.5}) or \( \ll \log W \) (by Lemma \ref{lem:3.4}). It is therefore safe to assume the estimate \( \ll \log^2 W \) for each zeta function. Finally, we have \( \frac{X^s}{s} \ll \frac{X}{W} \). Thus, the desired result follows.

\end{proof}

\begin{lemma}
\label{lem:4.7}

We have the estimate

\[
\int_{\ell_2} \frac{\zeta'}{\zeta}(s + iy_1) \zeta(s - i(y_2 - y_1)) \zeta(s) \frac{X^{s}}{s} \, ds \ll X^b \log^5 W. \tag{4.52} \label{eq:4.52}
\]

\end{lemma}

\begin{proof}
    
On \( \ell_2 \), using Lemma \ref{lem:3.4}, we obtain the bound \( \zeta(s) \ll \log |t| \) for \( |t| \gg 1 \). Furthermore, we find that \( \frac{\zeta'}{\zeta}(s) \ll \log |t| \cdot \log \log |t| \ll \log^2 t \) for \( |t| \gg 1 \). Additionally, we have \( \frac{X^{s}}{s} \ll X^b \min(1, t^{-1}) \) for all \( t \) on \( \ell_2 \). For small \( t \), by Lemma \ref{lem:3.5}, we know that \( \zeta(s) \ll \log W \), and a similar bound holds for \( \frac{\zeta'}{\zeta}(s) \). Moreover, \( \frac{X^{s}}{s} \ll 1 \) in this region. Since the interval of integration is finite, the additional term \( \ll \log^3 W \) can be absorbed into the case \( |t| \gg 1 \). The result follows.

\end{proof}

The calculation for line \( \ell_3 \) proceeds similarly to that for \( \ell_1 \). Thus, for the final integral, we have

\begin{lemma}
\label{lem:4.8}

\[
\int_{\ell_3} \frac{\zeta'}{\zeta}(s + iy_1) \zeta(s - i(y_2 - y_1)) \zeta(s) \frac{X^{s}}{s} \, ds \ll \frac{X}{W} \log^6 W. \tag{4.53} \label{eq:4.53}
\]

\end{lemma}
Putting these together, we obtain
\[
R_1(X, W) + R_2(X, W)
   \;\ll\; \frac{X}{W} \log^3 X
          \;+\; \frac{X}{W} \log^6 W
          \;+\; X^b \log^5 W,
\tag{4.54} \label{eq:4.54}
\]
where \(b = 1 - \dfrac{c}{\log W}\).  

We take
\[
W = \exp\!\big(C' \sqrt{\log T}\,\big),\qquad
x \asymp \log T,\qquad
X \asymp \frac{T}{x},
\]
and write \(L=\log T\), so that
\[
\log X \asymp L, \qquad \log W = C' \sqrt{L}.
\]

Substituting these into \eqref{eq:4.54}, the three contributions are
\[
E_1 := \frac{X}{W}\log^3 X 
   \;\asymp\; \frac{T}{L} \cdot \frac{L^3}{W}
   \;=\; T L^{2} e^{-C'\sqrt{L}},
\]
\[
E_2 := \frac{X}{W}\log^6 W
   \;\asymp\; \frac{T}{L} \cdot (C')^6 L^3 \cdot \frac{1}{W}
   \;=\; (C')^6 T L^{2} e^{-C'\sqrt{L}},
\]
\[
E_3 := X^{b}\log^5 W
   \;\asymp\; (T/L)^{b} (C')^5 L^{5/2}
   \;=\; (C')^5 T^{b} L^{5/2-b}.
\]

We choose the constants
\[
C' = C\!\left(2+\frac{\delta}{2}\right), 
\qquad 
c = (2+\delta)C^{2},
\]
where \(\delta>0\) is fixed and \(C\) is the constant in Theorem~\ref{thm:1.1}.  Then \(C'>2C\) and \(c/C'>C\).  Recalling that \(b=1-\dfrac{c}{C'\sqrt{L}}\), we have
\[
T^{b}=T^{1-c/(C'\sqrt{L})}
      = T\exp\!\left(-\frac{c}{C'}\sqrt{L}\right),
\]
and therefore
\[
E_3 \asymp T^{b}(\log T)^{5/2-b}
    = T(\log T)^{5/2-b}\exp\!\left(-\frac{c}{C'}\sqrt{L}\right).
\]
Similarly,
\[
E_1 \asymp T(\log T)^2 e^{-C'\sqrt{L}}.
\]
Since \(C'>C\),
\[
\frac{E_1 \log T}{T e^{-C\sqrt{L}}}
  \sim \log^3 T\,e^{-(C'-C)\sqrt{L}}
  \longrightarrow 0
  \qquad (T\to\infty).
\]
Moreover,
\[
\frac{E_3 \log T}{T e^{-C\sqrt{L}}}
  \asymp (\log T)^{\alpha}
      \exp\!\Big(-\Big(\frac{c}{C'}-C\Big)\sqrt{L}\Big)
  \longrightarrow 0,
\]
because
\[
\frac{c}{C'}-C
= C\Big(\frac{2+\delta}{2+\delta/2}-1\Big)
= \frac{C(\delta/2)}{2+\delta/2}>0.
\]
Hence \(E_1 \asymp E_2 \ll E_3= o \big(\frac{T e^{-C\sqrt{L}}}{\log T}\big)\).

Therefore, we have
\[
R_1(X,W)+R_2(X,W)\ll \frac{T e^{-C\sqrt{L}}}{\log T}.
\]

Since \(L'_1(X)=S_3(X)+R_1(X,W)+R_2(X,W)\), we obtain
\[
x^{-1+i y_1} L_1
= L'_1(X_2)-L'_1(X_1)
= S_3(X_2)-S_3(X_1)
 + (R_1(X_2,W)-R_1(X_1,W))
 + (R_2(X_2,W)-R_2(X_1,W)),
\]
where
\[
X_j=\frac{T_j+y_1}{2\pi x},\qquad 
T_1=T,\qquad
T_2=(1+\epsilon)T,\qquad 
\epsilon=\exp\!\big(-C\sqrt{\log T}\big).
\]

The function \(S_3(X)\) is a linear combination of \(X\), \(X^{1-i y_1}\), and \(X^{1+i(y_2-y_1)}\); hence \(S_3'(X)\asymp 1\).  
Therefore,
\[
S_3(X_2)-S_3(X_1)
= (X_2-X_1) S_3'(\xi)
\ll |X_2-X_1|
= \frac{T_2-T_1}{2\pi x}
= \frac{T\epsilon}{x}
\ll \frac{T e^{-C\sqrt{L}}}{\log T},
\]
using \(x\asymp \log T\). Finally, by multiplying \(L'_1(X_2)-L'_1(X_1)\) by \(x^{1-i y_1}\), whose modulus is \(x\asymp \log T\), we arrive at
\[
L_1 \ll T \exp\!\big(-C\sqrt{\log T}\big),\tag{4.55} \label{eq:4.55}
\]
which is negligible compared with the main term in Theorem~\ref{thm:1.1}.

\medskip

\noindent\textbf{Remark.} The estimate \(S_3'(X)\asymp 1\) used above is justified as follows.  
Each constituent term of \(S_3(X)\) is of the form \(X^{1+i\tau}\) (with \(\tau\in\{0,-y_1,y_2-y_1\}\) fixed); for \(X>0\) we have
\[
\frac{d}{dX} X^{1+i\tau}=(1+i\tau)X^{i\tau}.
\]
Hence, for \(X_2>X_1>0\),
\[
X_2^{1+i\tau}-X_1^{1+i\tau}
=\int_{X_1}^{X_2} (1+i\tau) u^{i\tau}\,du,
\]
and, since \(|u^{i\tau}|=1\) for real \(u>0\),
\[
\big|X_2^{1+i\tau}-X_1^{1+i\tau}\big|
\le |1+i\tau|\,|X_2-X_1|.
\]
As \(y_1,y_2\) are fixed, the constants \(|1+i\tau|\) are \(O(1)\); consequently
\[
S_3(X_2)-S_3(X_1)\ll |X_2-X_1|,
\]
so one may take \(S_3'(X)\asymp 1\) in the previous estimates.

\subsubsection{Bound on \( J_2 \)}

Next, we consider \( J_2 \). By Lemma \ref{lem:3.9}, we can write 

\[
J_2 = \frac{-x \log \pi}{2 \pi i} \sum_{n=1}^{\infty} \sum_{a_1a_2=n} \left( a_1^{iy_1} a_2^{iy_2} \right) \int_{c-iT_2}^{c-iT_1} \chi(1-s+iy_1) (nx)^{-s} \, ds := L_3 + O(L_4), \tag{4.56} \label{eq:4.56}
\]
where

\[
L_3 = -x^{1-iy_1} \log \pi \sum_{\frac{T_1 + y_1}{2 \pi x} < n \leq \frac{T_2 + y_1}{2 \pi x}} \sum_{a_1a_2=n} a_2^{i(y_2 - y_1)}, \tag{4.57} \label{eq:4.57}
\]
and

\[
L_4 = \sum_{n=1}^{\infty} \frac{d(n)}{n^c}  \ \times 
\]

\[
\left\{ (T_1 + y_1)^{c-\frac{1}{2}} + \frac{(T_1 + y_1)^{c+\frac{1}{2}}}{|T_1 + y_1 - 2 \pi n x| + (T_1 + y_1)^{1/2}} + \frac{(T_2 + y_1)^{c+\frac{1}{2}}}{|T_2 + y_1 - 2 \pi n x| + (T_2 + y_1)^{1/2}} \right\}. \tag{4.58} \label{eq:4.58}
\]

The estimate for \(L_4\) is analogous to that for \(L_2\); see \eqref{eq:4.46}. In particular, since \(L_4 \ll L_2\), we obtain
\begin{equation}
L_4 = o\!\left( T e^{-C\sqrt{\log T}} \right).
\tag{4.59}
\label{eq:4.59}
\end{equation}

We now turn to the estimation of $L_3$. We have

\begin{theorem}\label{thm:L3-bound}
Let \( x \asymp \log T \), where \( T = \frac{1}{2}(T_1 + T_2) \), and suppose \( y_1, y_2 \in \mathbb{R} \) are fixed. Define
\[
L_3 = -x^{1 - i y_1} \log \pi \sum_{\frac{T_1 + y_1}{2 \pi x} < a_1 a_2 \le \frac{T_2 + y_1}{2 \pi x}} a_2^{i(y_2 - y_1)}.
\]
Then we have
\[
L_3 \ll T \exp\!\big(-C\sqrt{\log T}\big).
\]
\end{theorem}
\begin{proof}[Sketch of proof]
The proof follows the same pattern as that of Theorem~\ref{lem:L1-bound}.
Define
\[
L'_3(T_2) := -\log \pi \sum_{a_1 a_2 \le \frac{T_2+y_1}{2\pi x}}
a_2^{i(y_2-y_1)},
\]
so that \(L_3 = x^{1-iy_1}(L'_3(T_2)-L'_3(T_1))\).
Let \(X=\frac{T_2+y_1}{2\pi x} \asymp T/\log T\).

The Dirichlet series associated with the coefficients
\(b_n=\sum_{a_2\mid n} a_2^{i(y_2-y_1)}\) is
\(\zeta(s)\zeta(s-i(y_2-y_1))\). Applying the truncated Perron formula
with the same choices \(W=\exp(C'\sqrt{\log T})\) and shifting the contour
to \(\Re s=b=1-\frac{c}{\log W}\) yields
\[
L'_3(T_2) = S_4(X) + O\!\Big(T e^{-C\sqrt{\log T}}/\log T\Big),
\]
where \(S_4(X)\) is a linear combination of \(X\) and \(X^{1+i(y_2-y_1)}\),
and the error term already incorporates the factor \(1/\log T\) coming from
\(X \asymp T/\log T\). An identical estimate holds for \(L'_3(T_1)\).

Since \(S_4'(X)\asymp1\), we have
\[
S_4(X_2)-S_4(X_1) \ll |X_2-X_1| 
= \frac{T_2-T_1}{2\pi x} 
\ll \frac{T e^{-C\sqrt{\log T}}}{\log T},
\]
where we used \(T_2-T_1 \asymp T e^{-C\sqrt{\log T}}\) and \(x\asymp\log T\).

Therefore,
\[
x^{-1+iy_1}L_3 = L'_3(T_2)-L'_3(T_1) \ll \frac{T e^{-C\sqrt{\log T}}}{\log T}.
\]
Multiplying by \(x^{1-iy_1}\) (whose modulus is \(\asymp\log T\)) gives
\[
L_3 \ll T e^{-C\sqrt{\log T}},
\]
as required.
\end{proof}

\subsubsection{Bound on \( J_3 \)}

We now proceed to \( J_3 \). By Lemma \ref{lem:3.9}, this can be written as

\[
J_3 = \frac{x^{1-iy_1}}{4 \pi i} \sum_{n=1}^{\infty} \sum_{a_1a_2=n} \left( a_1^{iy_1} a_2^{iy_2} \right) \int_{c-iT_2}^{c-iT_1} \psi\left( \frac{s}{2} \right) \chi(1-s+iy_1) (nx)^{-s} \, ds := L_5 + O(L_6), \tag{4.64} \label{eq:4.64}
\]
where

\[
L_5 = \frac{1}{2} x^{1-iy_1} \sum_{\frac{T_1 + y_1}{2 \pi x} < n \leq \frac{T_2 + y_1}{2 \pi x}} \left( \log(\pi n x) - \frac{\pi i}{2} \right) \sum_{a_1a_2=n} a_2^{i(y_2 - y_1)}, \tag{4.65} \label{eq:4.65}
\]
and

\[
L_6 = \sum_{n=1}^{\infty} \frac{d(n) \log (T_1 + y_1)}{n^c}  \ \times 
\]
\[
\left\{ (T_1 + y_1)^{c-\frac{1}{2}} + \frac{(T_1 + y_1)^{c+\frac{1}{2}}}{|T_1 + y_1 - 2 \pi n x| + (T_1 + y_1)^{1/2}} + \frac{(T_2 + y_1)^{c+\frac{1}{2}}}{|T_2 + y_1 - 2 \pi n x| + (T_2 + y_1)^{1/2}} \right\}. \tag{4.66} \label{eq:4.66}
\]

Note that \(L_6 \ll L_4 \log T\) and one might deduce \(L_6 = o\!\big(T \log T e^{-C\sqrt{\log T}}\big)\) from this, 
but we can obtain a better estimate as follows.

Separate \(L_6\) into three parts corresponding to the three terms in braces, i.e.,
\[
\begin{aligned}
L_6^{(1)} &= \log t_1 \; t_1^{c-\frac12} \sum_{n=1}^{\infty} \frac{d(n)}{n^c},\\[4pt]
L_6^{(2)} &= \log t_1 \; t_1^{c+\frac12} \sum_{n=1}^{\infty} 
           \frac{d(n)}{n^c\big(|t_1-2\pi n x|+t_1^{1/2}\big)},\\[4pt]
L_6^{(3)} &= \log t_1 \; t_2^{c+\frac12} \sum_{n=1}^{\infty} 
           \frac{d(n)}{n^c\big(|t_2-2\pi n x|+t_2^{1/2}\big)},
\end{aligned}
\]
where \(t_j = T_j + y_1 \asymp T\) and \(c = 1 + \frac{1}{\log x}\).

For the first part,
\[
\sum_{n=1}^{\infty} \frac{d(n)}{n^c} = \zeta^2(c) \ll (\log x)^2,
\]
hence
\[
L_6^{(1)} \ll \log T \; T^{c-\frac12} (\log x)^2.
\]

For the second and third parts, a bound analogous to Lemma~\ref{lem:3.18} but without the 
extra \(\log n\) factor gives
\[
\sum_{n\ge 1} \frac{d(n)}{n^c\big(|t-2\pi n x|+t^{1/2}\big)} 
   \ll \frac{\log^2 t + \log^2 x}{t},
\]
so that
\[
L_6^{(2)}, L_6^{(3)} \ll \log T \; T^{c+\frac12} \cdot \frac{\log^2 T}{T}
      = \log^3 T \; T^{c-\frac12}.
\]

Adding the three contributions and recalling \(x \asymp \log T\) (so \(\log x \asymp \log\log T\))
yields
\[
L_6 \ll T^{c-\frac12}\big(\log^3 T + \log T (\log\log T)^2\big)
     \ll T^{\frac12 + \frac{1}{\log\log T}} \log^3 T. \tag{4.67} \label{eq:4.67}
\]

Finally,
\[
\frac{L_6}{T e^{-C\sqrt{\log T}}}
\ll T^{-\frac12 + \frac{1}{\log\log T}} \log^3 T \, e^{C\sqrt{\log T}}
\to 0 \qquad (T\to\infty),
\]
hence \(L_6 = o\!\big(T e^{-C\sqrt{\log T}}\big)\) as desired.

We now calculate \( L_5 \). Since \( \log(\pi n x) - \frac{\pi i}{2} = \log \frac{T_1 + T_2}{2} + O(1) \) for all \( n \) in the sum, we have

\[
L_5 = \frac{1}{2} x^{1 - i y_1} \left( \log T + O(1) \right) \sum_{\frac{T_1 + y_1}{2 \pi x} < a_1 a_2 \le \frac{T_2 + y_1}{2 \pi x}} a_2^{i(y_2 - y_1)}. \tag{4.68} \label{eq:4.68}
\]

We will observe that the main terms of \( L_5 \) and \( L_7 \) (see \eqref{eq:4.81}) are equal, as are the corresponding error terms.

Next, we evaluate \( L_5' \), defined as follows

\[
L_5' = L_5'(T_2) := \sum_{\substack{a_1, a_2 \\ a_1 a_2 \le \frac{T_2 + y_1}{2 \pi x}}} a_2^{i(y_2 - y_1)}. \tag{4.69} \label{eq:4.69}
\]

Multiplying \( L_5'(T_2) - L_5'(T_1) \) by \( \frac{1}{2} x^{1 - i y_1} \left( \log \frac{T_1 + T_2}{2} + O(1) \right) \) yields \( L_5 \). In what follows, we use the notation \( X := \frac{T_2 + y_1}{2\pi x} \).

To estimate \(L_5'\), one could in principle apply the method of hyperbola
summation (see, for example, \cite{Apostol2013}, Theorem~3.17). However, this
approach leads to a highly complicated expression, and it appears difficult
to verify directly that the resulting main term is nonzero. Instead, we
proceed as in the preceding cases and estimate \(L_5'\) by applying
Lemma~\ref{lem:4.1}.

In this case, the coefficients \( b_n \) satisfy
\[
|b_n|
= \left| \sum_{\substack{a_1,a_2 \\ a_1 a_2 = n}} a_2^{i(y_2-y_1)} \right|
\le \sum_{\substack{a_1,a_2 \\ a_1 a_2 = n}} 1
= d(n).
\tag{4.70}
\label{eq:4.70}
\]
By calculations analogous to those leading to \eqref{eq:4.20}, we obtain, in the present case, the bound
\[
R_1(X,W) \ll \frac{X}{W}\log^2 X.
\]

To evaluate the integral in Lemma \ref{lem:4.1}, we consider the rectangular contour defined by the following path descriptions.

\begin{itemize}
  \item[\(\ell_1\):] \( [b - Wi, a - Wi], \)
  \item[\(\ell_2\):] \( [b + Wi, b - Wi], \)
  \item[\(\ell_3\):] \( [a + Wi, b + Wi], \)
  \item[\(\ell_4\):] \( [a - Wi, a + Wi]. \)
\end{itemize}

Similar to the two previous cases, we have \( X \geq e \), \( a = 1 + \frac{1}{\log X} \), and \( \frac{1}{2} < b = 1 - \frac{c}{\log W} < 1 \). Moreover, \( W \) must be sufficiently large to accommodate the shift \( y_2 - y_1 \). 

The function

\[
h_X(s) = \zeta(s - i(y_2 - y_1)) \zeta(s) \frac{X^s}{s} \tag{4.71} \label{eq:4.71}
\]
has simple poles at \( s = 1 + i(y_2 - y_1) \) and \( s = 1 \).

Note that the poles of \( h_X(s) \) are distinct, and the sum of the residues is given by

\[
\sum \text{Res} \left( \zeta(s - i(y_2 - y_1)) \zeta(s) X^{s} \frac{1}{s} \right) = 
\]

\[
\zeta(s) X^{s} \frac{1}{s} \bigg|_{s = 1 + i(y_2 - y_1)}
+ \zeta(s - i(y_2 - y_1)) X^{s} \frac{1}{s} \bigg|_{s = 1}
\]

\[
= \zeta(1 + i(y_2 - y_1)) X^{1 + i(y_2 - y_1)} \frac{1}{1 + i(y_2 - y_1)} + \zeta(1 - i(y_2 - y_1)) X := S_4(X). \tag{4.72} \label{eq:4.72}
\]

By the Cauchy Integral Formula, we have

\[
\int_{\ell_4} h_X(s)  ds - S_4(X) = - \int_{\ell_1} h_X(s)  ds - \int_{\ell_2} h_X(s)  ds - \int_{\ell_3} h_X(s)  ds := R_2(X, W). \tag{4.73} \label{eq:4.73}
\]

The integrals are evaluated using arguments similar to those for the sum over \( \ell_1 \), yielding the bounds
\begin{align*}
\int_{\ell_1} \zeta(s - i(y_2 - y_1)) \zeta(s) \frac{X^{s}}{s}  ds &\ll \frac{X}{W} \log^4 W, \tag{4.74} \label{eq:4.74} \\
\int_{\ell_2} \zeta(s - i(y_2 - y_1)) \zeta(s) \frac{X^{s}}{s}  ds &\ll X^b \log^3 W, \tag{4.75} \label{eq:4.75} \\
\int_{\ell_3} \zeta(s - i(y_2 - y_1)) \zeta(s) \frac{X^{s}}{s}  ds &\ll \frac{X}{W} \log^4 W. \tag{4.76} \label{eq:4.76}
\end{align*}
We omit the detailed proofs as they are similar to those already given.

Putting all of this together, we obtain
\[
R_1(X, W) + R_2(X, W) \ll \frac{X}{W} \log^2 X + \frac{X}{W} \log^4 W + X^b \log^3 W. \tag{4.77} \label{eq:4.77}
\]
We now choose
\[
W = \exp\!\big(C'\sqrt{\log T}\big)
\]
with the same constant \( C' = C(2+\delta/2) \) as in Theorem~\ref{lem:L1-bound}, and recall that
\(
X \asymp T/x \asymp T/\log T.
\)
Comparing with the estimates for \(L_1\) in \eqref{eq:4.54}, we see that each term in the bounds for \(L_5\) has smaller logarithmic powers: the power of \(\log X\) is \(\log^2 X\) versus \(\log^3 X\), the second term has \(\log^4 W\) versus \(\log^6 W\), and the third term \(\log^3 W\) versus \(\log^5 W\). Each of these differences contributes an additional factor of \((\log T)^{-1}\) relative to \eqref{eq:4.54}. Taken together, the total error in \eqref{eq:4.77} satisfies
\[
R_1(X,W)+R_2(X,W)
\ll
\frac{T e^{-C\sqrt{L}}}{\log^2 T}.
\]

Therefore,
\[
L_5'(T_2)
=
S_4(X)
+
O\!\Big(\frac{T e^{-C\sqrt{\log T}}}{\log^2 T}\Big).
\]

Therefore, the integral over \(\ell_4\) in \eqref{eq:4.73} equals \(S_4(X) + R_1(X, W) + R_2(X, W)\), which consequently implies 

\[
L_5 = \frac{1}{2} x^{1 - i y_1} \left( \log T + O(1) \right) \left( S_4\left( \frac{T_2 + y_1}{2\pi x} \right) - S_4\left( \frac{T_1 + y_1}{2\pi x} \right) + O\!\Big(\frac{T e^{-C\sqrt{\log T}}}{\log^2 T}\Big) \right). \tag{4.78} \label{eq:4.78}
\]

Replacing \( S_4 (\cdot ) \) from \eqref{eq:4.72} into \eqref{eq:4.78}, we finally obtain

\[
L_5 = \frac{1}{2} x^{1 - i y_1} \left( \log T + O(1) \right) 
\left( \zeta(1 + i(y_2 - y_1)) \left( \frac{T_2 + y_1}{2 \pi x} \right)^{1 + i(y_2 - y_1)} \frac{1}{1 + i(y_2 - y_1)} \right.
\]

\[
\left. + \zeta(1 - i(y_2 - y_1)) \left( \frac{T_2 + y_1}{2 \pi x} \right)
- \zeta(1 + i(y_2 - y_1)) \left( \frac{T_1 + y_1}{2 \pi x} \right)^{1 + i(y_2 - y_1)} \frac{1}{1 + i(y_2 - y_1)} \right.
\]

\[
\left. - \zeta(1 - i(y_2 - y_1)) \left( \frac{T_1 + y_1}{2 \pi x} \right) + O\!\Big(\frac{T e^{-C\sqrt{\log T}}}{\log^2 T}\Big) \right). \tag{4.79} \label{eq:4.79}
\]

Before proceeding further, we need to estimate \( J_4 \) from \eqref{eq:4.41}.

\subsubsection{Bound on \( J_4 \)}

Next, we consider the sum

\[
J_4 = \frac{x^{1-iy_1}}{4 \pi i} \sum_{n=1}^{\infty} \sum_{a_1a_2=n} \left( a_1^{iy_1} a_2^{iy_2} \right) \int_{c-iT_2}^{c-iT_1} \psi\left( \frac{1-s}{2} \right) \chi(1-s+iy_1) (nx)^{-s} \, ds. \tag{4.80} \label{eq:4.80}
\]

This sum can be decomposed as follows:

\[
J_4 = L_7 + O(L_8),
\]
where

\[
L_7 = \frac{1}{2} x^{1-iy_1} \sum_{\frac{T_1 + y_1}{2 \pi x} < n \leq \frac{T_2 + y_1}{2 \pi x}} \left( \log(\pi n x) + \frac{\pi i}{2} \right) \sum_{a_1a_2=n} a_2^{i(y_2 - y_1)}, \tag{4.81} \label{eq:4.81}
\]
and

\[
L_8 = \sum_{n=1}^{\infty} \frac{d(n) \log (T_1 + y_1)}{n^c}  \ \times 
\]
\[
\left\{ (T_1 + y_1)^{c-\frac{1}{2}} + \frac{(T_1 + y_1)^{c+\frac{1}{2}}}{|T_1 + y_1 - 2 \pi n x| + (T_1 + y_1)^{1/2}} + \frac{(T_2 + y_1)^{c+\frac{1}{2}}}{|T_2 + y_1 - 2 \pi n x| + (T_2 + y_1)^{1/2}} \right\}. \tag{4.82} \label{eq:4.82}
\]

Since \( L_8 = L_6 \) by \eqref{eq:4.66}, we have the estimate

\[
L_8 = o\!\big(T e^{-C\sqrt{\log T}}\big). \tag{4.83} \label{eq:4.83}
\]

Note that the only difference between \( L_7 \) in \eqref{eq:4.81} and \( L_5 \) in \eqref{eq:4.68} is the change in the sign of \( \pi i /2 \); therefore, the same calculations apply. Summing these two contributions and performing straightforward simplifications, we write
\[
L_5 + L_7 = M_1 + M_2 + M_3,
\]
where

\[
M_1 = x^{1 - i y_1} \left( \log \frac{T_1 + T_2}{2} + O(1) \right) \frac{\zeta(1 + i(y_2 - y_1))}{1 + i(y_2 - y_1)} \left[ \left( \frac{T_2 + y_1}{2 \pi x} \right)^{1 + i (y_2 - y_1)} - \left( \frac{T_1 + y_1}{2 \pi x} \right)^{1 + i (y_2 - y_1)} \right]
\]

\[
= x^{1 - i y_1} \log T \frac{\zeta(1 + i(y_2 - y_1))}{1 + i(y_2 - y_1)} \left[ \left( \frac{T_2 + y_1}{2 \pi x} \right)^{1 + i (y_2 - y_1)} - \left( \frac{T_1 + y_1}{2 \pi x} \right)^{1 + i (y_2 - y_1)} \right] + O\!\big(T e^{-C\sqrt{\log T}}\big); \tag{4.84} \label{eq:4.84}
\]

\[
M_2 = x^{1 - i y_1} \left( \log \frac{T_1 + T_2}{2} + O(1) \right) \zeta(1 - i(y_2 - y_1)) \left( \frac{T_2 - T_1}{2 \pi x} \right)
\]

\[
= \frac{x^{-i y_1}}{2 \pi} \log T \cdot \zeta(1 - i(y_2 - y_1)) (T_2 - T_1) + O\!\big(T e^{-C\sqrt{\log T}}\big); \tag{4.85} \label{eq:4.85}
\]
and

\[
M_3 = x^{1 - i y_1} \left( \log \frac{T_1 + T_2}{2} + O(1) \right) \cdot O\Big(\frac{T e^{-C\sqrt{\log T}}}{\log^2 T}\Big)
= O\!\big(T e^{-C\sqrt{\log T}}\big); \tag{4.86} \label{eq:4.86}
\]
where we used the asymptotic estimates \( x \asymp \log T \) and \( T_1 \asymp T_2 \asymp T := \frac{1}{2}(T_1 + T_2) \).

The \(O\)-terms for \(M_2\) and \(M_3\) are clear. We briefly explain how the error in \(M_1\) arises. It comes from the \(O(1)\) term in the factor 
\(\log\frac{T_1+T_2}{2}+O(1)\).  

Let \(\alpha := y_2 - y_1\) and define \(X_j := \frac{T_j + y_1}{2\pi x}\) for \(j=1,2\).  
Since \(X_j \asymp T/x \asymp T/\log T\), the difference
\[
X_2^{1+i\alpha} - X_1^{1+i\alpha}
\]
can be estimated using the mean value theorem:
\[
\bigl|X_2^{1+i\alpha} - X_1^{1+i\alpha}\bigr|
\ll |X_2 - X_1|\;\max_{\xi\in[X_1,X_2]}|\xi^{i\alpha}|
\ll \frac{\Delta}{x}
\asymp \frac{T e^{-C\sqrt{\log T}}}{\log T},
\]
where we used \(\Delta = T_2 - T_1 \asymp T e^{-C\sqrt{\log T}}\) and \(x \asymp \log T\).  

Multiplying this bound by \(x^{1-iy_1} \asymp \log T\) and by the constant factor \(\zeta(1+i\alpha)/(1+i\alpha)\) produces a contribution of size
\[
O\!\Bigl(\log T \cdot \frac{T e^{-C\sqrt{\log T}}}{\log T}\Bigr)
= O\!\bigl(T e^{-C\sqrt{\log T}}\bigr).
\]

Up to this point, we have shown that every term arising from \(I_1\) to \(I_4\) in \eqref{eq:4.1}--\eqref{eq:4.4} is \(O(T \log^{1/2} T \, e^{-C\sqrt{\log T}})\), except for the main contributions of \(M_1 + M_2\). Therefore, establishing the nonvanishing of the sum in \eqref{eq:2.1} reduces to showing that the main term of \(M_1 + M_2\) is nonzero, which we address in the next section.

\subsection{The Main Term}
\label{sec:4.5}
Let's denote the main term of \( M_1 + M_2 \) in \eqref{eq:4.84} and \eqref{eq:4.85} by \( M \). Thus, we have

\[
M = x^{1 - i y_1} \log T \frac{\zeta(1 + i(y_2 - y_1))}{1 + i(y_2 - y_1)} \left[ \left( \frac{T_2 + y_1}{2 \pi x} \right)^{1 + i (y_2 - y_1)} - \left( \frac{T_1 + y_1}{2 \pi x} \right)^{1 + i (y_2 - y_1)} \right]
\]

\[
+ \frac{x^{-i y_1}}{2 \pi} \log T \cdot \zeta(1 - i(y_2 - y_1)) (T_2 - T_1), \tag{4.87} \label{eq:4.87}
\]
which, after some work, transforms to

\[
M = (T_2 - T_1) \log T \frac{1}{2 \pi} x^{- i y_1} \zeta(1 - i(y_2 - y_1)) \times
\]

\[
\left[ \frac{\zeta(1 + i (y_2 - y_1))}{\zeta(1 - i (y_2 - y_1))} \left( \frac{T_1 + y_1}{2 \pi x} \right)^{i (y_2 - y_1)} \frac{1}{1 + i (y_2 - y_1)} \left( \frac{T_2 + y_1}{T_2 - T_1} \left( \frac{T_2 + y_1}{T_1 + y_1} \right)^{i (y_2 - y_1)} - \frac{T_1 + y_1}{T_2 - T_1} \right) + 1 \right]
\]

\[
:= (T_2 - T_1) \log T \frac{1}{2 \pi} x^{- i y_1} \zeta(1 - i(y_2 - y_1)) M'(x, y_1, y_2). \tag{4.88} \label{eq:4.88}
\]

Therefore, to prove the nonzeroness of the sum in \eqref{eq:2.1} and, thus, prove our main theorem, it is sufficient to show that the expression in the square brackets, i.e., \( M' = M'(x, y_1, y_2) \) in \eqref{eq:4.88}, is bounded below by some positive constant.

To begin, we examine the difference between two \( M'(x, y_1, y_2) \) values, evaluated at two distinct primes \( x_2 \) and \( x_1 \), where these values of \( x_1 \) and \( x_2 \) will be determined later. We then obtain the following expression

\[
|M'(x_2, y_1, y_2) - M'(x_1, y_1, y_2)| = \left| \frac{x_2^{-i (y_2 - y_1)} - x_1^{-i (y_2 - y_1)}}{1 + i (y_2 - y_1)} \left( \frac{T_2 + y_1}{T_2 - T_1} \left( \frac{T_2 + y_1}{T_1 + y_1} \right)^{i (y_2 - y_1)} - \frac{T_1 + y_1}{T_2 - T_1} \right) \right|, \tag{4.89} \label{eq:4.89}
\]
where we applied the Schwarz reflection principle for the zeta function, i.e., \( \overline{\zeta(s)} = \zeta(\overline{s}) \), which implies \( \left| \frac{\zeta(1 + i (y_2 - y_1))}{\zeta(1 - i (y_2 - y_1))} \right| = 1 \).

By setting \( T_2 = (1 + \epsilon)T_1 \), where $\epsilon = \exp\left(-C \sqrt{\log T_1}\right)$ for any constant \( C > 0 \), the expression inside the parentheses on the right-hand side of \eqref{eq:4.89} for sufficiently large \( T_1 \) becomes

\[
\frac{T_2 + y_1}{T_2 - T_1} \left( \frac{T_2 + y_1}{T_1 + y_1} \right)^{i (y_2 - y_1)} - \frac{T_1 + y_1}{T_2 - T_1}
\]

\[
= \left( \frac{T_1(1 + \epsilon) + y_1}{T_1 \epsilon} \left( \frac{T_1(1 + \epsilon) + y_1}{T_1 + y_1} \right)^{i (y_2 - y_1)} - \frac{T_1 + y_1}{T_1 \epsilon} \right)
\]

\[
= \left( \frac{T_1(1 + \epsilon) + y_1}{T_1 \epsilon} \exp\left( i (y_2 - y_1) \log \left( \frac{T_1(1 + \epsilon) + y_1}{T_1 + y_1} \right) \right) - \frac{T_1 + y_1}{T_1 \epsilon} \right)
\]

\[
= 1 + i (y_2 - y_1) + i (y_2 - y_1) \epsilon + i (y_2 - y_1) \frac{y_1}{T_1} + o(1). \tag{4.90} \label{eq:4.90}
\]

Substituting this result back into \eqref{eq:4.89}, we then obtain 

\[
|M'(x_2, y_1, y_2) - M'(x_1, y_1, y_2)| = \left| x_2^{-i (y_2 - y_1)} - x_1^{-i (y_2 - y_1)} \right| (1 + o(1)). \tag{4.91} \label{eq:4.91}
\]
Assuming \( y_1 \neq y_2 \), we define \( \alpha := e^{\pi / (4|y_2 - y_1|)} \). By Lemma 2.1 of \cite{Banks2024}, for sufficiently large \( t \), there exist primes \( x_1 \in (t, \alpha t) \) and \( x_2 \in (\alpha^3 t, \alpha^4 t) \) such that  
\[
\left| x_2^{-i (y_2 - y_1)} - x_1^{-i (y_2 - y_1)} \right| > \sqrt{2}.
\]  
We take \( t = \log T_1 \), which is why we assumed in previous sections the estimate \( x \asymp \log T \), where \( T = \frac{1}{2}(T_1 + T_2) \asymp T_1 \asymp T_2 \).  
This implies that for at least one of the \( x_i \), we have  
\[
|M'(x, y_1, y_2)| > \frac{1}{\sqrt{2}} (1 + o(1)), \tag{4.92} \label{eq:4.92}
\]  
and thus, the proof is complete.


\appendix
\section{Asymptotic Bounds and Parameter Growth}  

\newcounter{aplemma}
\renewcommand{\theaplemma}{A.\arabic{aplemma}}
\newenvironment{applemma}[1][]{%
  \refstepcounter{aplemma}\par\noindent\textbf{Lemma~\theaplemma. #1} \itshape
}{\par}


This appendix collects several lemmas that will be used in Section~4 
to estimate terms arising in the evaluation of the integrals in 
\eqref{eq:4.1}--\eqref{eq:4.4}.

\medskip
\begin{applemma}\label{lem:A1}
Let \( X \asymp xT \) with \( x \asymp \log T \), so \( X \asymp T \log T \).  
Let \( b = 1 - \frac{c}{\log(xW)} \) for some constant \( c > 0 \), small enough so that \( b > \frac{1}{2} \), and define
\[
S = \frac{X}{W} \log^3 X + \frac{X}{W} \log^6 W + X^b \log^5 W. \tag{A.1} \label{eq:A.1}
\]
Then there exists no increasing function \( W = W(T) \) such that
\[
S = o\left( \log T \cdot T^{1 - \frac{k}{\log \log T}} \right)
\quad \text{for all } k > 0. \tag{A.2} \label{eq:A.2}
\]
\end{applemma}

\noindent\textbf{Remark.} The purpose of this lemma is to demonstrate that the choice \( \varepsilon = T^{-k / \log \log T} \), as defined in Theorem~1.1 of Banks' paper, cannot be used effectively in the two-variable shift setting.

\begin{proof}

We show that for every increasing \( W = W(T) \), there exists some \( k > 0 \) such that
\[
S \neq o\left( \log T \cdot T^{1 - \frac{k}{\log \log T}} \right).
\]

To this end, we consider two cases: $\log W \gg \log \log T$ and $\log W \ll \log \log T$. In each case, we analyze the individual terms in $S$.

First suppose $\log W \gg \log \log T$. Then, for the first term we have 

\[
\frac{X}{W} \log^3 X \asymp \frac{T \log T}{W} (\log T)^3 = \frac{T (\log T)^4}{W}.
\]

By requiring 
\[
\frac{T (\log T)^4}{W} = o\left( \log T \cdot T^{1 - \frac{k}{\log \log T}} \right),
\]
and dividing both sides by \( T \log T \) we obtain
\[
\frac{(\log T)^3}{W} = o\left( T^{-\frac{k}{\log \log T}} \right).
\]

Taking logarithms gives
\[
3 \log \log T - \log W \ll -\frac{k \log T}{\log \log T}.
\]
Thus,
\[
\log W \gg \frac{k \log T}{\log \log T}.
\]

Similarly, for the second term we obtain 
\[
\frac{(\log W)^6}{W} = o\left( T^{-\frac{k}{\log \log T}} \right).
\]
Taking logarithms gives
\[
6 \log \log W - \log W \ll -\frac{k \log T}{\log \log T}.
\]
Assuming \( \log W \gg \log \log T \), \( \log W \) dominates and we have
\[
-\log W \ll -\frac{k \log T}{\log \log T} \implies \log W \gg \frac{k \log T}{\log \log T}.
\]
This is consistent with the first term.

Now we analyze the third term. Recall that
\[
b = 1 - \frac{c}{\log(xW)} \asymp 1 - \frac{c}{\log(\log T \cdot W)} = 1 - \frac{c}{\log \log T + \log W}.
\]
Let \( A_T = \log \log T + \log W \). We have \(
X^b \asymp (T \log T)^b = (T \log T)^{1 - \frac{c}{A_T}}.
\)
Thus,
\[
X^b \log^5 W \asymp (T \log T)^{1 - \frac{c}{A_T}} (\log W)^5 = T \log T \cdot (\log W)^5 \cdot \exp\left( -\frac{c \log(T \log T)}{A_T} \right).
\]

Since \( \log(T \log T) \sim \log T \),
\[
X^b \log^5 W \asymp T \log T \cdot (\log W)^5 \cdot \exp\left( -\frac{c \log T}{A_T} \right).
\]
Requiring
\[
T \log T \cdot (\log W)^5 \cdot \exp\left( -\frac{c \log T}{A_T} \right) = o\left( \log T \cdot T^{1 - \frac{k}{\log \log T}} \right),
\]
and divide both sides by \( T \log T \) and taking logarithms gives
\[
5 \log \log W - \frac{c \log T}{A_T} \ll -\frac{k \log T}{\log \log T}.
\]

Case \( \log W \gg \log \log T \) means that \( \log W \) grows at least as fast as \( \log \log T \) and thus \( A_T \sim \log W \). Therefore,
\[
\frac{c \log T}{A_T} \sim \frac{c \log T}{\log W}.
\]
The condition becomes
\[
5 \log \log W - \frac{c \log T}{\log W} \ll -\frac{k \log T}{\log \log T}.
\]
For large \( T \), the dominant term is \( -\frac{c \log T}{\log W} \). For this to be much less than \( -\frac{k \log T}{\log \log T} \), we need
\[
\frac{c \log T}{\log W} \gg \frac{k \log T}{\log \log T}.
\]
This implies
\[
\log W \ll \frac{c \log \log T}{k}.
\]
But from the first two terms, we have
\[
\log W \gg \frac{k \log T}{\log \log T}.
\]
Putting these two together, for large \( T \), we get
\[
\frac{k \log T}{\log \log T} \ll \frac{c \log \log T}{k},
\] 
which is impossible. Therefore, no \( W \) can satisfy both inequalities for all \( k > 0 \).

Next, we consider the second case, i.e., when $\log W \ll \log \log T$. Imposing this case on each term of $S$ leads to a contradiction, but for completeness, we present all of them.

Since, we have $3 \log \log T - \log W \sim 3 \log \log T$, so applying condition~\eqref{eq:A.2} for the first term in~\eqref{eq:A.1} implies 

$$
3 \log \log T \ll -\frac{k \log T}{\log \log T}.
$$

But this is false, since the left hand side is  $O(\log \log T)$ and the right hand side $-\Theta\left( \frac{\log T}{\log \log T} \right)$, negative and much larger in magnitude.

For the second term, since under the case $\log W \ll \log \log T$, we have $6 \log \log W - \log W \gg -\log \log T$, again by applying condition~\eqref{eq:A.2}  we have

\[
-\log \log T \ll -\frac{k \log T}{\log \log T}.
\]

This is again is impossible for large $T$, as the LHS is much larger (less negative) than RHS. So the inequality fails.

And finally, for the third term we have \( A_T \sim \log \log T \), which results in
\[
\frac{c \log T}{A_T} \sim \frac{c \log T}{\log \log T}.
\]
Therefore, by~\eqref{eq:A.2} we have

\[
5 \log \log W - \frac{c \log T}{\log \log T} \ll -\frac{k \log T}{\log \log T}.
\]
For large \( T \), as \( \log \log W =O(\log T) \), so the dominant term is \( -\frac{c \log T}{\log \log T} \).  
This implies
\[
-\frac{c \log T}{\log \log T} \ll -\frac{k \log T}{\log \log T},
\]
which is equivalent to
\[
c \gg k.
\]
But this must hold for all \( k > 0 \), which is impossible since \( c \) is fixed.  
Thus, for any \( k > 0 \), the condition fails.

This completes the proof.

\end{proof}

\begin{applemma}
\label{lem:A2}
Let \( X \), \( b \) and \( S \) be as defined in Lemma~\ref{lem:A1}. Then the function
\[
g(T) = T \exp\left(-C \sqrt{\log T}\right),
\]
for some constant \( C > 0 \), is the asymptotically minimal function for which there exists a function \( W = W(T) \) such that
\[
S = o\left( \log T \cdot g(T) \right). \tag{A.3} \label{eq:A.3}
\]
A possible choice of \( W \) for Condition~\eqref{eq:A.3}
to hold is 
\[
W = \exp\left(C' \sqrt{\log T}\right),
\]
for some constant \( C' > 0 \).

\end{applemma}

\begin{proof}

Decomposing the expression for \( S \), and assuming \( X \asymp T \log T \), the terms become
\[
S \asymp \frac{T (\log T)^4}{e^{A_T}} + \frac{T \log T \cdot A_T^6}{e^{A_T}} + T \log T \cdot A_T^5 \exp\left(-\frac{c \log T}{A_T} \right), \tag{A.4} \label{eq:A.4}
\]
where \( A_T = \log W \).

To ensure that \( S = o(\log T \cdot g(T)) \), it is sufficient that each of the terms in \( S \) is \( o(\log T \cdot g(T)) \). Let \( g(T) = T e^{-h(T)} \). 

For the first term in~\eqref{eq:A.4}, we have
\[
\frac{T (\log T)^4}{e^{A_T}} = o(T \log T \cdot e^{-h(T)}) \quad \Rightarrow \quad (\log T)^3 e^{-A_T} = o(e^{-h(T)}),
\]
which implies \( e^{-A_T} = o(e^{-h(T)}) \). Taking logarithms gives \( h(T) \ll A_T \).

For the second term, we have
\[
\frac{T \log T \cdot A_T^6}{e^{A_T}} = o(T \log T \cdot e^{-h(T)}) \quad \Rightarrow \quad A_T^6 e^{-A_T} = o(e^{-h(T)}),
\]
which again implies \( e^{-A_T} = o(e^{-h(T)}) \), so \( h(T) \ll A_T \).

For the third term, we have
\[
T \log T \cdot A_T^5 \exp\left(-\frac{c \log T}{A_T} \right) = o(T \log T \cdot e^{-h(T)}) \quad \Rightarrow \quad A_T^5 \exp\left(-\frac{c \log T}{A_T} \right) = o(e^{-h(T)}).
\]
Taking logarithms yields
\[
5 \log A_T - \frac{c \log T}{A_T} \ll -h(T) \quad \Rightarrow \quad \frac{c \log T}{A_T} \gg h(T) + 5 \log A_T \Rightarrow \quad 
A_T \ll \frac{c \log T}{h(T)}.
\]

Combining the two inequalities for \( A_T \), namely \( A_T \gg h(T) \) from the first and second terms, and \( A_T \ll \frac{c \log T}{h(T)} \) from the third term in~\eqref{eq:A.4}, we get
\[
h(T) \ll \frac{c \log T}{h(T)} \quad \Rightarrow \quad h(T)^2 \ll c \log T.
\]

The minimal value of \( g(T) \) occurs when equality holds, that is, when \( h(T) \asymp \sqrt{c \log T} \), yielding
\[
g(T) = T \exp\left(-C \sqrt{\log T}\right),
\]
for some constant \( C > 0 \).

To test a suitable choice of \( W \), let us take \( A_T = \log W \asymp \sqrt{\log T} \).  
With \( g(T) = T \exp(-C \sqrt{\log T}) \), the first two terms in~\eqref{eq:A.4} satisfy
\[
T (\log T)^4 e^{-k \sqrt{\log T}}
 \;+\;
T \log T \, k^6 (\log T)^3 e^{-k \sqrt{\log T}}
 \;\asymp\;
T (\log T)^4 e^{-k \sqrt{\log T}},
\]
for some \( k>0 \).  
Imposing the requirement that this be \( o\!\left( T \log T \, e^{-C \sqrt{\log T}} \right) \) yields
\[
T (\log T)^4 e^{-k \sqrt{\log T}}
= o\!\left( T \log T \, e^{-C \sqrt{\log T}} \right),
\]
which holds if and only if \( k > C \).

A similar analysis for the third term yields
\[
T \log T \cdot k^5 (\log T)^{5/2} \exp\left(-\frac{c \log T}{k \sqrt{\log T}}\right) = T k^5 (\log T)^{7/2} \exp\left(-\frac{c}{k} \sqrt{\log T}\right).
\]
Making the last term to be equal to $o(T \log T \cdot e^{-C \sqrt{\log T}})$ and simplifying further we arrive at the condition
\[
k^5 (\log T)^{5/2} \exp\left( \left(C - \frac{c}{k}\right) \sqrt{\log T} \right) \to 0 \quad \text{as } T \to \infty,
\]
and this holds if and only if \( C < \frac{c}{k} \).

There exists a value of \( k \) satisfying both \( k > C \) and \( C < \frac{c}{k} \) if and only if \( C^2 < c \). In this case, valid values of \( k \) lie in the interval \( (C, \frac{c}{C}) \). Thus, choosing \(k\) within this interval ensures that all required conditions are satisfied, and therefore the second statement of the lemma is proved.

To prove the minimality claim, that is to show that asymptotically \( g(T) = T \exp(-C \sqrt{\log T}) \) is the smallest possible function for which  Condition~\eqref{eq:A.3}  holds, suppose by contradition \( g_1(T) = T e^{-h_1(T)} \) with \( h_1(T) = \omega(\sqrt{\log T}) \) (i.e., \( h_1(T) \) grows faster than \( \sqrt{\log T} \)), also satisfies Condition~\eqref{eq:A.3}. Then from the third term of $S$ and imposing the condition, we have
\[
A_T \ll \frac{c \log T}{h_1(T)} = o(\sqrt{\log T}),
\]
while from the first two terms, we have 
\[
A_T \gg h_1(T) = \omega(\sqrt{\log T}),
\]
which is a contradiction. Therefore, no such \( g(T) \) can work, and \( g(T) = T \exp(-C \sqrt{\log T}) \) is indeed the asymptotically minimal function.

\end{proof}

\medskip

\noindent\textbf{Remark.}
Banks~\cite{Banks2024} proved his one-variable result for intervals of the form
\(\epsilon = T^{1 - \frac{k}{\log \log T}}\). However, Lemma~\ref{lem:A1} shows that choosing such an interval in our two-variable setting makes
equation~\eqref{eq:4.24} insufficiently small compared to the main term of equation~\eqref{eq:2.3}.
To overcome this issue, one may either assume GRH or enlarge the interval; the latter is the approach taken in Lemma~\ref{lem:A2}.
In that lemma, we also show that no smaller interval can suffice. To compare these two intervals, note that Banks' choice may be written in the form
\[
T^{1-\frac{k}{\log\log T}}
= T\exp\!\big(-(\log T)^{a(T)}\big),
\qquad 
a(T)=1+\frac{\log k-\log\log\log T}{\log\log T}=1-o(1),
\]
whereas Lemma~\ref{lem:A2} yields an interval of the form 
\(T\exp\!\big(-C(\log T)^{1/2}\big)\).
Thus the exponent in our setting is the fixed power \(1/2\), while Banks' effective
exponent satisfies \(a(T)\to1\); in particular, his interval decays much more rapidly and is therefore too small to be useful in the two-variable shift setting.

\section*{Acknowledgments}

I sincerely thank Dr Bryce Kerr for his guidance and constructive suggestions. His mathematical insight and encouragement were instrumental in completing this paper.

\end{document}